\documentclass[10pt,reqno]{amsart}
\usepackage{amsmath, amsthm, amssymb, stmaryrd}
\usepackage{hyperref}
\usepackage{enumerate}
\usepackage{url}
\usepackage{color}
\usepackage{tikz}
\usepackage{wrapfig}

\let\OLDthebibliography\thebibliography
\renewcommand\thebibliography[1]{
  \OLDthebibliography{#1}
  \setlength{\parskip}{0pt}
  \setlength{\itemsep}{0pt plus 0.3ex}
}

\usepackage{mathtools}

\usepackage{multirow, array}
\usepackage{placeins}

\usepackage{caption} 
\advance \topmargin by -\headheight
\advance \topmargin by -\headsep
     
\evensidemargin \oddsidemargin
\newtheorem{maintheorem}{Theorem}

\newtheorem{thm}{Theorem}[section]
\newtheorem{thm*}[thm]{Theorem$^\star$}

\newtheorem{lemma}[thm]{Lemma}
\newtheorem{prop}[thm]{Proposition}

\theoremstyle{definition}
\newtheorem{defn}[thm]{Definition}

\theoremstyle{remark}
\newtheorem{remark}[thm]{Remark}

\numberwithin{equation}{section}

\newcommand*\wrapletters[1]{\wr@pletters#1\@nil}
\def\wr@pletters#1#2\@nil{#1\allowbreak\if&#2&\else\wr@pletters#2\@nil\fi}

\newcommand{\dd}{\; \mathrm{d}}

\def\le{\leqslant} \def\ge{\geqslant}

\def \bN {\mathbb N}
\def \bP {\boldsymbol{P}}
\def \bR {\boldsymbol{R}}
\def \bQ {\boldsymbol{Q}}

\def \bbR {\mathbb R}
\def \bZ {\mathbb Z}

\def \bh {\mathbf h}

\def \bxi {\mathbf \xi}

\def \bx {\mathbf x}

\def \bxi {{\boldsymbol{\xi}}}

\def \cC {\mathcal C}

\def \cF {\mathcal F}

\def \cL {\mathcal L}
\def \cM {\mathcal M}

\def \cP {\mathcal P}
\def \cQ {\mathcal Q}
\def \cR {\mathcal R}

\def \cT {\mathcal T}

\def \dim {\mathrm{dim}}

\def \Bad {{\mathrm{Bad}}}

\begin{document}
\title[]{Bad approximability, bounded ratios and Diophantine exponents}
\author[A. Marnat, N. Moshchevitin, J. Schleichitz]{Antoine Marnat \\ {\small\sc Universit\'e Paris Est - Cr\'eteil } \\ \\ Nikolay Moshchevitin \\ {\small\sc
{ 
Technische Universit\"at Wien
}
 }\\ \\ Johannes Schleichitz \\ {\small\sc Middle East Technical University}   }

\subjclass[2010]{11J83}
\keywords{Diophantine approximation, geometry of numbers}
\thanks{The second named author is supported by Austrian Science Fund (FWF), Forschungsprojekt PAT1961524}

\date{}
\begin{abstract}  

For a real {$n\times m$} matrix $\pmb{\xi}$, we consider its sequence of best Diophantine approximation vectors $ \pmb{x}_i \in \mathbb{Z}^{{m}}, \, i =1,2,3, ...  $, the sequences of its norms $X_i = \|\pmb{x}_i\|$ and the norms of remainders $L_i = \|\pmb{\xi}\pmb{x}_i\|$. It is  known that, in the cases  $m=1$,  bad approximability of $\pmb{\xi}$ is equivalent to the boundedness of ratios $\frac{X_{i+1}}{X_i}$, while for $n=1$   bad approximability of $\pmb{\xi}$ is equivalent to the boundedness of ratios $ \frac{L_i}{L_{i+1}}$. Moreover,
carefully constructed example show that in the cases $m=1$ and $n=1$ boundedness of ratios $ \frac{L_i}{L_{i+1}}$ and  $\frac{X_{i+1}}{X_i}$ respectively (the order of ratios changed), does not imply bad approximability of $\pmb{\xi}$. In the present paper, we study the impact of the boundedness of ratios on Diophantine properties of  $\pmb{\xi}$, in particular, what restrictions it gives for Diophantine exponents $\omega(\pmb{\xi})$ and $\hat{\omega}(\pmb{\xi})$. One of our particular results deals with the case $m=n=2$. We prove that for $2\times 2 $  matrices  $\pmb{\xi}$  boundedness of both ratios  $ \frac{X_{i+1}}{X_i}, \frac{L_i}{L_{i+1}}  $  implies inequality  $\hat{\omega}(\pmb{\xi})\le \frac{4}{3}$ and that this result is optimal. Our methods combine parametric geometry of numbers as well as more classical tools.
\end{abstract}

\maketitle

\section{Diophantine approximation and bad approximability}\label{beginn}

Let  $m, n\ge1$ be integers. Throughout the paper, $\| \cdot \|$ denotes the sup norm, and we use the Vinogradov symbol $A \ll_k B$ and $A\gg_k B$ if there exists a positive constant $c:=c(k)$ depending on $k$ such that $A\le c B$ and $A \ge cB$ respectively.\\
We will deal with integer vectors of the form
$$\bx =  (\pmb{x},\pmb{y})^\top = (x_1, \ldots , x_m, y_1, \ldots , y_n)^\top \in \bZ^{m+n},$$
$$
  \pmb{x} = (x_1, \ldots , x_m)^\top\in \mathbb{Z}^m,\,\,\,   \pmb{y} = (y_1, \ldots , y_n)^\top\in \mathbb{Z}^n.
$$
 
We consider Diophantine approximation to the matrix $\bxi = (\xi_{i,j})_{1\le i \le n, 1\le j \le m}\in \bbR^{n \times m}$. 

Denote 
\begin{equation}\label{defL}
L_{\bxi}(\bx) := \| \pmb{\xi}\pmb{x} -\pmb{y}\| =  \max_{1\le i \le n} \|x_1\xi_{i,1} + \cdots + x_m\xi_{i,m} -y_i \|.
\end{equation}

The classical Dirichlet Theorem reads as follows.

\begin{maintheorem}[Dirichlet, 1842]
For every $X>1$, there exists integers $\bx \in \bZ^{m+n}\,\,\,$ such that
\begin{equation}\label{Dir}
\|\pmb{x}\| \le X^n, \; L_{\bxi}(\bx)\le X^{-m}.
\end{equation}
\end{maintheorem}

  \vskip+0.3cm
 We will consider only \emph{totally irrational} $\bxi$, that is  $L_{\bxi}(\bx)>0$ for every $\bx \in \bZ^{m+n}\setminus \{\boldsymbol{0}\}$
 and the subspace 
 $$\{ \bx \in \bbR^{m+n} \mid L_\bxi (x) = 0 \}$$
  is not contained in any rational subspace $Q \subset \bbR^{n+m}$ of dimension $\dim(Q) < n+m$. 
 
 \vskip+0.3cm
 
 In particular, Dirichlet implies that  the inequality $$\|\pmb{x}\|^m L_{\bxi}(\bx)^n \le 1$$ holds for infinitely many integer $ \bx\in \mathbb{Z}^{m+n}$. A matrix  $\bxi$ is called  \emph{badly approximable} if  there exists a positive constant $c$ such that  the inequality $$\|\pmb{x}\|^m L_{\bxi}(\bx)^n \ge c$$ holds for all non-zero integer vectors $\bx\in \mathbb{Z}^{m+n}$. We denote by $\Bad_{m,n}$ the set of badly approximable $\bxi \in \bbR^{n \times m}$.
 For  more discussion on this notion of bad approximability and further references we refer the reader to  the  classical book \cite{SSS} as well as to a recent paper
\cite{Yorkies}. 

In this paper, we focus on the characterization of badly approximable matrices in terms of sequences of \emph{best approximations}. Best approximation vectors were considered by many authors. We refer to a survey paper \cite{Che} and a recent paper \cite{Mo2} where all the aspects of the definition are discussed in details. 
Let  $(\bx_i)_{i\in \bN} \in \bZ^{m+n} $  be the  sequence of {best approximation}  vectors to $\bxi$. If we set $X_i = \|\pmb{x}_i\|$ and $L_i = L_{\bxi}(\bx_i)$, then
\begin{equation}\label{Xsec}
1 = X_1< X_2< \cdots<X_\nu < \cdots \quad \text{ and }\quad L_1 > L_2 > \cdots > L_\nu>\cdots
 \end{equation}
are strictly monotone. By definition, for all $\|\pmb{x}\| < X_i$, we have $L_{\bxi}(\bx) \ge L_{i-1}$. 

 {As we only consider totally irrational $\bxi$, this sequences are infinite.}  Depending on dimensions $m$ and $n$ and the matrix $\bxi$, the sequence $(\bx_i)_{i\in \bN}$ may not be unique, but the sequences $(X_i)_{i\in \bN}$ and $(L_i)_{i\in \bN}$ are always defined uniquely (see discussion in \cite{Mo2}). To analyse the quality of Diophantine approximation, we may consider for a real parameter $w$ and every $X$ the system of inequalities
\begin{equation}\label{DefOmega}
 \begin{cases}
L_{\bxi}(\bx) < X^{-w}\cr
  0< \|\pmb{x}\| \le X
\end{cases}  .
\end{equation}
 The ordinary exponent of Diophantine approximation $\omega(\bxi)$ is defined to be the supremum over  all  $w$ such that \eqref{DefOmega} has solutions for arbitrarily large $X$, and the uniform exponent of Diophantine approximation $\hat{\omega}(\bxi)$  is defined as the supremum of all $w$ such that \eqref{DefOmega} has solutions for all sufficiently large $X$. Note that Dirichlet's Theorem ensures that $\frac{m}{n} \le \hat{\omega}(\bxi) \le  \omega(\bxi)$. In terms of the sequence of best approximation, the ordinary and uniform exponents of Diophantine approximation can be given by equalities
\begin{equation}\label{defexp}
 \omega(\bxi) =   \limsup_{i \to \infty} \frac{- \log L_i}{\log X_{i}} ,\; \quad \hat{\omega}(\bxi) =  \liminf_{i \to \infty} \frac{- \log L_i}{\log X_{i+1}}.
\end{equation}
Finally, a matrix  $\bxi$ is called  \emph{singular} if $$\limsup_{i\to\infty}X_{i+1}^m L_i^n = 0.$$ 
When $\hat{\omega}(\bxi) > \frac{m}{n}$, then $\bxi$ is clearly singular. In the latter case $\bxi$ is called \emph{supersingular}. Various aspects of singularity are discussed for example in \cite{Mo1}. Many classical results related to the topic can be found in books by Cassels \cite{CCC} and Schmidt \cite{SSS}.

In the next section, we study how the notions of bad approximability and singularity relate by ratios of successive best approximations.

\section{Ratios for successive best approximations: main results } 

Consider the following three Diophantine properties of a  {totally irrational} matrix $\bxi\in \bbR^{n\times m}$.

 \medskip
\begin{enumerate}[(A)]
\item{} $\bxi\in \bbR^{n\times m}$ is badly approximable.

\vskip+0.3cm
 
\item{} the ratio $X_{i+1} / X_i$ is bounded, that is $\sup_{i\in \bN} \frac{X_{i+1}}{X_i}< \infty$.
 
 \vskip+0.3cm
\item{} the ratio $L_i / L_{i+1}$ is bounded, that is $\sup_{i\in \bN} \frac{L_{i}}{L_{i+1}}< \infty$.
\end{enumerate}

\medskip

In \cite{AM}, R. Akhunzhanov and N. Moshchevitin ask how these conditions relate to each other. When the dimension is $m=n=1$, the theory of continued fractions shows that these three conditions are equivalent. For simultaneous approximation ($m=1$), Akhunzhanov and Moshchevitin \cite{AM} showed that (A) and (B) are equivalent and imply (C), but that for $n=2$, (C) does not imply (A) or (B) in return. W. M. Schmidt \cite{S} generalizes this fact to any $n\ge 2$. As for dual results ($n=1$),  Akhunzhanov and Moshchevitin \cite{AM} proved that for $n=1$ that (A) and (C) are equivalent and imply (B). Dual ($n=1$) existence results that (B) does not imply (A) were just announced in \cite{AM}.  Further results on this problem were obtained by L. Summerer \cite{L} studying admissible uniform exponents, see Section \ref{sss4}. The aim of the present paper is to study possible  generalizations of these observations to arbitrary $m$ and $n$.\\

\subsection{Outline of the rest of the section}
In the present section, we formulate all the results related to  properties (A), (B) and (C) and their quantitative versions (with respect to ordinary or uniform exponents of Diophantine approximation) as theorems. We divide Section 2 into five subsections.\\

First we explain in Subsection ~\ref{sss01} that the notions and properties related to the sequence of best approximations under consideration do not depend on the choice of the norm. 
In {Subsection} \ref{sss1} we will formulate all qualitative results about implications. In Theorem \ref{Thm-} we formulate well-known  implications from \cite{AM}. Theorem \ref{nm3}  gathers results which show that there is no converse implication and claims the existence of corresponding examples.
In {Subsection}  \ref{sss2},  in the low-dimensional case $m+n=3$, we give new sufficient conditions for $\pmb{\xi}$ to be non-singular.
In {Subsection}  \ref{sss3} we formulate quantitative results about possible values of the ordinary Diophantine exponent $\omega=\omega(\pmb{\xi})$ for matrices $\pmb{\xi}$ satisfying properties (B) and/or (C). 
In {Subsection}  \ref{sss4} we deal with possible values of the uniform Diophantine exponent $\hat{\omega} =\hat{\omega}(\pmb{\xi})$ for matrices $\pmb{\xi}$ satisfying properties (B) and/or (C). 
In  {Subsection} \ref{OpenProblems} we gather the relevant open questions which remain unsolved.
At last, in {Subsection} \ref{sss5}  
we give  the plan of the rest of the paper and explain for every formulated result in what part of the paper it is proven.\\

We briefly highlight the novelty of results  of theorems from our paper, for the reader's convenience.
\textbf{Theorem} \ref{normindep} is easy and rather technical, but it is likely that it has been never documented before.
The results of \textbf{Theorem} \ref{Thm+} are not new. They are taken from  the paper \cite{AM} by Akhunzhanov and Moshchevitin.
However we give an alternative proof completely in terms of parametric geometry of numbers, and this  was not done earlier 
for all the  implications from this theorem. 
All the statements of \textbf{Theorem}$^\star$~\ref{Thm-} for $m\ge 2, n\ge 2$ are new.
Cases $n=1$ and $m=1$  are not new and were known due to Akhunzhanov and Moshchevitin \cite{AM} and Schmidt \cite{S}, however the proof for  
the case 
$n=1$ has never been documented before.
The result of \textbf{Theorem} \ref{nm3} was not observed before.
All the results of \textbf{Theorem}$^\star$\ref{ThmOrd} are completely new.
The results of \textbf{Theorem} \ref{OptiOrdi} are rather simple but new.
The results of \textbf{Theorem}$^\star$~\ref{Main} for $n\ge 2$ are new. The case $n=1$ was obtained earlier by Summerer \cite{L}.
\textbf{Theorem}~\ref{optimn} is completely new.

\vskip+0.3cm
{
Here we should make one important note. Existence \textbf{Theorems}$^\star$~\ref{Thm-}, \ref{ThmOrd} and \ref{Main} of matrices with certain Diophantine properties rely on the \emph{variational principle in parametric geometry of numbers} \cite{David}. The proof of this variational principle uses a special dynamical process and specific versions of Schmidt's games. In fact, it does not provide explicit constructions. Among specialists in Diophantine Approximation, the general opinion is that the proof is difficult to understand (see \cite{Wgene}, page 390). Actually, this approach builds bridges between Diophantine and dynamical worlds.  In this manuscript, we emphasise the Theorems depending on the variational principle in parametric geometry of numbers with the symbol $\star$. Our proofs use the variational principle as a black box.}

\subsection{Independence of the norm} \label{sss01} 
Before formulating our main result, we  give a comment concerning properties  (A), (B)  and (C). In the previous section, we define the notion of bad approximability and the sequences  of best approximations \eqref{Xsec} with respect to the sup-norm $\|\cdot \|$. However for these definitions, one can use an arbitrary pair of norms $\|\cdot\|_1$ in $\mathbb{R}^m$ and
$\|\cdot\|_2$ in $\mathbb{R}^n$. This leads to new sequences of best approximations
\begin{equation}\label{Xsec1}
1 = X_1' = \| \pmb{x}_1\|_1< X_2'< \cdots< X_\nu' < \cdots \,\,\,\,\, \text{ and } \,\,\,\,\,  L_1' =\|\pmb{\xi} \pmb{x}_1 - \pmb{y}_1\|_2> L_2' > \cdots >L_\nu'> \cdots
 \end{equation}
which depend on norms $\|\cdot\|_1, \|\cdot\|_2$ and may differ from (\ref{Xsec}). Sequences (\ref{Xsec1}) are also defined uniquely.
 On the other hand, by equivalence of the norms in $\bbR^d$, a matrix $\pmb{\xi}$ is badly approximable if and only if for any pair of norms
$\|\cdot\|_1, \|\cdot\|_2$  the inequality
 $$\|\pmb{x}\|_1^m \,  \| \pmb{\xi}\pmb{x} -\pmb{y}\|_2^n \ge c'$$ 
 holds for all non-zero integer vectors $ 
 \bx\in \mathbb{Z}^{m+n}$ with some positive constant  $c'$ depending on the pair of norms under consideration. In this sense, the definition of bad approximability does not depend on the choice of the norms $\|\cdot\|_1, \|\cdot\|_2$. It turns out that the same situation happens with the ratios
$X_{i+1} / X_i$,
 $L_i / L_{i+1}$ 
from the properties (B) and (C). This fact can be documented as the following statement.

\begin{thm}\label{normindep}
Consider a {totally irrational} matrix $\bxi\in \bbR^{n\times m}$
and an arbitrary pair of norms
$\|\cdot\|_1$ in $\mathbb{R}^m$ and
$\|\cdot\|_2$ in $\mathbb{R}^n$.
Then 
\begin{itemize}
\item  the ratio  $X_{i+1} / X_i$ is bounded if and only if the ratio $X_{i+1}' / X_i'$ is bounded,\\
\item  the ratio   $L_i / L_{i+1}$  is bounded if and only if the ratio  $L_i' / L_{i+1}'$  is bounded.
\end{itemize}
\end{thm}

So, the three properties (A), (B) and (C) are independent from the choice of the two norms.
Similarly, the values of Diophantine exponents  ${\omega}(\bxi)$ and $\hat{\omega}(\bxi)$ do not depend on the choice of  norms.

\subsection{ Qualitative results}\label{sss1}

This first theorem here recalls known results from \cite{AM} related to implications that hold between the properties (A), (B) and (C).
\begin{thm}\label{Thm+} Consider a { totally irrational} matrix $\bxi \in \bbR^{n \times m}$.
\begin{enumerate}[\rm1.]
\item For any $m,n\ge 1$, (A) implies (B) and (C).\\
\item For $m=1$, (B) implies (A) and (C). In particular, (A) and (B) are equivalent.\\
\item For $n=1$, (C) implies (A) and (B). In particular, (A) and (C) are equivalent.
\end{enumerate}
\end{thm}

This theorem  was originally proved using classical methods of geometry of numbers.  In fact, Statement 1. is obvious. 
We present a complete alternative proof of Statements 2. and 3. in the context of \emph{parametric geometry of numbers}. This context proves to be very adapted to our matter: for example in Section \ref{beau} we provide an existence result using a very simple and elegant multigraph.

Our second theorem states that implications which are not mentioned in Theorem \ref{Thm+} do not hold.

\begin{thm*}\label{Thm-} \hfill
\begin{enumerate}[\rm1.]
\item For $m,n\ge 2$, there exists { totally irrational} $\bxi \in \bbR^{n \times m}$ such that
\begin{equation}\label{thm-1}
\bxi \textrm{ has properties (B) and (C) but not (A)}.
\end{equation}
\item For $m\ge1, n\ge2$, there exists {totally irrational} $\bxi \in \bbR^{n \times m}$ such that
\begin{equation}\label{thm-2}
\bxi \textrm{ has property (C) but neither (A) nor (B)}.
\end{equation}
\item For $m\ge2, n\ge1$, there exists { totally irrational} $\bxi \in  \bbR^{n \times m}$ such that
\begin{equation}\label{thm-3}
\bxi \textrm{ has property (B) but neither (A) nor (C)}.
\end{equation}

\item Furthermore, the sets of { totally irrational} $\bxi \in \bbR^{n \times m}$ satisfying respectively \eqref{thm-1}, \eqref{thm-2} and \eqref{thm-3} each have full Hausdorff dimension $mn$.
\end{enumerate}
\end{thm*}

For $m=1$, case $n=2$ is done in \cite{AM} and the case $n\ge 2$ is done in \cite{S} using parametric geometry of numbers. Case $n=1$  was  only announced  both in \cite{AM} ($m=2$) and \cite{S} ($m\ge2$). When both  $m$ and $n$ are at least $2$ the results are new, as well as the result on Hausdorff  dimensions.\\

 Before we state more precise results on these (counter)examples, studying their admissible ordinary and uniform exponents,  we should observe that this means  that conditions (B) and (C) are not relevant to characterize badly approximable matrices when $\min(m,n)>1$.\\

\subsection{$m+n=3$: (non)singularity}\label{sss2}

In the special case $m+n=3$, we can observe the following  properties regarding singularity.

\begin{thm}\label{nm3} Consider a {totally irrational} $\bxi\in\bbR^{n \times m}$ with $m+n=3$.
\begin{enumerate}[\rm1.]
\item If $n=2$ and $m=1$
and property (C) holds, then $\bxi\in\bbR^{{ 2\times 1}}$ is not singular.
\vskip+0.3cm
\item If $n=1$ and $m=2$ 
and property (B) holds,
then $\bxi\in\bbR^{{ 1\times2}}$ is not singular.
\end{enumerate}
\end{thm}

Here we should note that by the famous Khintchine's theorem, the column vector $\pmb{\xi} \in \mathbb{R}^{2\times 1} $ is singular if and only if the row vector $\pmb{\xi}^\top \in \mathbb{R}^{1\times 2}$ is singular, so in the case under consideration properties (B) or (C) ensure non-singularity for the transposed vector as well.

In particular, if we suppose that   $\bxi\in\bbR^{2\times 1}$ satisfies (C) or  $\bxi^\top \in\bbR^{1\times 2}$ satisfies (B), one has $\hat{\omega}(\bxi) = \frac{1}{2}$ and  $\hat{\omega}(\bxi^\top) = 2$.

Note that Theorem ~\ref{nm3} cannot be extended readily to higher dimension $m+n>3$, see remark after Theorem$^\star$ ~\ref{Main}.

\vskip+0.3cm

\begin{remark}
We have mentioned above that in \cite{S}  Schmidt considered the case $ m=1, n \ge 2$ and constructed $\pmb{\xi}$ such that (C) is satisfied but (A) is not satisfied. Here we would like to point out that in Schmidt's construction from \cite{S}, the vector $\pmb{\xi}$ has uniform exponent $\frac{1}{n}$ thus is not \emph{supersingular}.
\end{remark}

\subsection{About admissible ordinary exponent}\label{sss3}

The following theorem states the possible values of the ordinary exponent of Diophantine approximation of the (counter)examples constructed in Theorem \ref{Thm-}.

\begin{thm*}\label{ThmOrd} For dimensions $m,n\ge 1$, consider the sets
\begin{eqnarray*}
S_{B,n,m}(\omega) &:=& \{\bxi \in \bbR^{n \times m} \mid \textrm{ $\bxi$ {is totally irrational},  satisfies (B) but not (A)}, \;  \omega(\bxi) = \omega \},\\
S_{C,n,m}(\omega) &:=& \{\bxi \in \bbR^{n \times m} \mid \textrm{ $\bxi$ {is totally irrational}, satisfies (C) but not (A)},  \; \omega(\bxi) = \omega \}.
\end{eqnarray*}
\begin{enumerate}[\rm1.]
\item If $m\ge2$ and $\omega \in [\frac{m}{n} , \infty]$, the set $S_{B,n,m}(\omega)$ is non-empty.\\
\item If $n\ge 2$ and $\omega' \in [\frac{m}{n} , m]$, the set $S_{C,n,m}(\omega')$ is non-empty.\\
\item If $m,n\ge 2$ and  $\omega' \in [\frac{m}{n} , m]$, the set $S_{B,n,m}(\omega') \cap S_{C,n,m}(\omega')$ is non-empty.\\
\end{enumerate}
\end{thm*}

\begin{remark}
It follows from the proof of Theorem$^\star$~\ref{ThmOrd} that when the sets $S_{B,n,m}(\omega)$, $S_{C,n,m}(\omega)$ or their intersection $S_{B,n,m}(\omega) \cap S_{C,n,m}(\omega)$  are non-empty they contain singular matrices.  We also obtain effective strictly positive lower bounds for their Hausdorff dimension: we apply a variational principle in the context of the parametric geometry of numbers, used for the construction. However precise cumbersome computations do not seem relevant here. We can also obtain full packing dimension in most cases, this is specified along the proof. Recall that Theorem$^\star$~\ref{Thm-} asserts that $S_{B,n,1}(\omega)= S_{C,1,m}(\omega)= \varnothing$.
\end{remark}

We should comment on the optimality of the admissible intervals for $\omega$. The interval of the first statement is optimal, as it is the full interval of admissible value for the ordinary exponent without further condition. This observation also implies that the lower bound $\frac{m}{n}$ in statements 2 and 3 is optimal. We conjecture that the upper bound $m$ when condition (C) holds is optimal. We prove it in the cases $m=1,2$, however for $m\ge 3$ we are only able to prove that the optimal upper bound is finite.

\begin{thm}\label{OptiOrdi}\label{neulemma} Consider a {totally irrational} matrix $\bxi \in \bbR^{n \times m}$.
\begin{enumerate}[\rm1.]
\item If $n\ge 2$ and $m=1$, any $\bxi \in \bbR^{{ n\times 1}}$ satisfying (C) has $\omega(\bxi) \le 1$.
\vskip+0.3cm
\item If $n\ge 2$ and $m=2$, any $\bxi \in \bbR^{{ n\times 2}}$ satisfying (C) has $\omega(\bxi) \le 2$.
\vskip+0.3cm
\item If $n\ge 2 $ and $m \ge 3$, any $\bxi\in \bbR^{n \times m}$ satisfying (C) has $\omega(\bxi) < \infty$.
\end{enumerate}
\end{thm}

\subsection{About admissible uniform exponent}
\label{sss4}

For $m=1$, L. Summerer \cite{L} shows that $\bxi\in\bbR^{1\times n}$ from \eqref{thm-2} can be chosen to have any  given uniform exponent $\hat{\omega}\in [\frac{1}{n}, \frac{1}{2} ]$, and that this interval is best possible, see Proposition 2 from \cite{AM}. We extend this observation to arbitrary  dimension.

\begin{thm*}\label{Main} 
Let $ n \ge 1$.
For the matrices $\bxi \in \bbR^{n \times m}$ considered in Theorem$^\star$~\ref{Thm-}, we can reach the following uniform exponents. 
\begin{enumerate}[(i)]
\item{Property (B) but not (A)}
\begin{enumerate}[]
\item{If $m\ge3$, then for any $\hat{\omega} \in [\frac{m}{n}, \infty]$ there exists  $\bxi\in \bbR^{n \times m}$ such that (B) holds but  (A) does not hold and $\hat{\omega}(\bxi) = \hat{\omega}$.}
\end{enumerate}

\item{Property (B) and (C) but not (A)}
\begin{enumerate}[a)]
If $\min(m,n)\geq 2$, then for any $\hat{\omega} \in [\frac{m}{n}, \frac{m^2}{m+1}]$ there exists  $\bxi\in \bbR^{n \times m}$ such that both (B) and (C) hold but  (A) does not hold and $\hat{\omega}(\bxi) = \hat{\omega}$.
\end{enumerate}

\end{enumerate}
\end{thm*} 

We now discuss the optimality of these intervals. Clearly (i) is optimal as the interval is the full admissible interval without constraint. The next theorem provides some bounds for (ii) 
{
in the  case $m=2$. In this case we approximate by integer points two-dimensional linear subspace
\begin{equation}\label{frakel}
 \frak{L} =  \frak{L}_{\pmb{\xi}}=\{\bx = (\pmb{x},\pmb{y})^\top,\,\, \pmb{x} \in \mathbb{R}^{2},\,\, 
\pmb{y} \in \mathbb{R}^{n}
: \,\,\, \pmb{y} = \pmb{\xi}\pmb{x}\}.
\end{equation}
Notice that  it may happen that subspace $\frak{L}_{\pmb{\xi}}$ belongs to a certain rational subspace.
 To deal with the case $m=2$ we need the following condition
 (which in fact is valid for totally irrational $\pmb{\xi}$):
 \begin{equation}\label{conti}
 \frak{L}_{\pmb{\xi}} \,\,\,\,\,\text{ does not belong to any three-dimensional rational subspace}.
\end{equation}
}

\begin{thm}\label{optimn} Fix $m=2$ and $n\ge1$, we have the following results.

\begin{enumerate}[\rm1.]
\item{} If { totally irrational} $\bxi\in \bbR^{{ n\times 2}}$ satisfies property (B) or property (C), then $\hat{\omega}(\bxi) \le 2$.
\item{} 
If {totally irrational 
} $\bxi\in \bbR^{{ n\times 2}}$ 
{(and so satisfying condition (\ref{conti})})
satisfies properties (B) and (C) simultaneously, then $\hat{\omega}(\bxi) \le 4/3$.
\end{enumerate}
\end{thm}

{
\begin{remark}
Statement 2 of Theorem \ref{optimn} points out  an additional assertion concerning matrix $\pmb{\xi}$. We should note that in the degenerate case when 
 $ \frak{L}_{\pmb{\xi}}$ belongs to a three-dimensional rational subspace, the situation is the same as in the case of one linear form in 
 two variables 
($m=2,n=1$) and property (C) leads to badly approximability of the corresponding linear form and equality  $\hat{\omega}(\bxi)=\omega(\pmb{\xi}) = 2$.
\end{remark}
}
\begin{remark}

The inequality  from Statement 2. of Theorem~\ref{optimn} shows that Statement (ii) of Theorem$^\star$~\ref{Main} is optimal for $m=2$.

Note that  (ii) of Theorem$^\star$~\ref{Main} in particular shows that 
\vskip+0.1cm
\noindent
a) For $m=2$ and for any $\hat{\omega} \in [\frac{2}{n}, \frac{4}{3}]$ there exists { totally irrational}  $\bxi\in \bbR^{{n\times 2}}$ such that (B) holds but (A) does not hold and $\hat{\omega}(\bxi) = \hat{\omega}$;
\vskip+0.1cm
and combined with Summerer \cite{L} ($m=1$)
\vskip+0.1cm
\noindent
b)  For $m\ge 1$ and $n>1$, and for any $\hat{\omega} \in [\frac{m}{n}, \frac{m^2}{m+1}]$, there exists { totally irrational} $\bxi \in \bbR^{n\times m}$ such that (C) holds but  (A) does not  and $\hat{\omega}(\bxi) = \hat{\omega}$. 
\vskip+0.1cm
We do not know 
 if a) and b) here  give optimal results.
\end{remark}

Again, regarding Hausdorff dimension, if the uniform exponent is finite, a strictly positive but cumbersome lower bound for the sets considered in Theorem$^\star$~\ref{Main} can be deduced from their construction. It appears to be bounded above at least by the dimension of the set of $2$-singular matrices (see Lemma \ref{lem}) and the set of supersingular matrices of levelset $\hat{\omega} > m/n$. The latters are not precisely known - see \cite{David} for more details, including a definition of $k$-singularity. We can also obtain information on the packing dimension.\\

\begin{remark}\label{RmkAlg} In the previous results, each time the constructed sets have full packing dimension $mn$, we can assume that the coordinates of the matrix $\bxi \in \bbR^{n \times m}$ are not only $\mathbb{Q}$-linearly independent but also algebraically independent. Indeed, the set of $\bxi\in \bbR^{n \times m}$ with algebraically dependent entries has packing dimension $mn - 1 < mn$, hence there is a full dimensional complement set of $\bxi$ left to choose from.

\end{remark}

\begin{remark}
Theorem$^\star$~\ref{Main} provides supersingular points satisfying both (B) and (C), proving that Theorem \ref{nm3} does not extend readily to $m+n>3$. 
\end{remark}

\subsection{Open problems}\label{OpenProblems} \

We gather some problems of interest that remain open.\\
\paragraph{\textbf{Problem 1:}} For $m\ge 3$ and $n\ge1$, find an upper bound for the ordinary exponent $\omega$ of matrices {totally irrational} $\bxi\in \bbR^{n \times m}$ satisfying (C).
We conjecture that the optimal bound is $m$.\\

\paragraph{\textbf{Problem 2:}} For $m\ge 2$ and $n\ge1$, find an upper bound for the uniform exponent $\hat{\omega}$ of { totally irrational} matrices $\bxi\in \bbR^{n \times m}$ satisfying (C).
We conjecture that the optimal bound is $\frac{m^2}{m+1}$.\\

 For both Problem 1 \& 2, one may first provide upper bounds for the subset of { totally irrational} matrices $\bxi\in \bbR^{n \times m}$ satisfying both (B) and (C). This consideration and phenomena observed in our construction in section \ref{qq4} leads to\\

\paragraph{\textbf{Problem 3:}} Are the upper bounds for ordinary and/or uniform exponents the same for matrices {totally irrational} $\bxi\in \bbR^{n \times m}$ satisfying only (C) and { totally irrational} matrices $\bxi\in \bbR^{n \times m}$ satisfying both (B) and (C)?\\

Observe that in our construction of families of templates in Section ~\ref{Construction}, the first component has non-extremal slope on arbitrarily long intervals. In other words, for $\bxi$ with successive minima function at bounded distance, the sequence of best approximation remains in some lower dimensional subspace for an arbitrarily long time. With this observation, one can tackle the following problem.\\

\paragraph{\textbf{Problem 4:}} Construct examples of { totally irrational} matrices $\bxi\in \bbR^{n \times m}$ for Theorem$^\star$~\ref{Thm-}, Theorem$^\star$~\ref{ThmOrd} and Theorem$^\star$~\ref{Main} without the tools of parametric geometry of numbers (theory of templates).\\

In particular, one probably does not need the full strength of the variational principle.
{
We are confident that \textbf{Theorem}$^\star$~\ref{Thm-} can be proven by a direct explicit construction, using no specific tools from parametric geometry of numbers.
This gives one more motivation to \textbf{Problem 4}.}

\subsection{Structure of the rest of the paper and some methods}\label{sss5}  \
  
\vskip+0.1cm 

Section  \ref{PGN} is devoted to the methods of parametric geometry of numbers. In Subsection \ref{123ss}  we explore Schmidt's observation \cite{S} how properties (B) and (C) can be translated to the language of parametric geometry of numbers. Subsection ~\ref{123sss} deals with a brief description of the variational principle  \cite{David} and the methods of its application. In Subsection \ref{123ssss}  we provide  proofs of some technical results dealing with bad approximability and boundedness of ratios in terms of parametric geometry of numbers.  
 In Subsection  \ref{Indepnorm} we prove {Theorem} \ref{normindep} which states that property of bounded ratios is independent of the norms. This easy proof uses only basic concepts of parametric geometry of numbers related to Schmidt's observation.
Section \ref{noRoy} is devoted to the proofs of those of our results, which  do  not involve  any application of the variational principle from \cite{David}.
In Subsection \ref{stpl} we give a proof of Theorem \ref{Thm+}   purely in terms of parametric geometry of numbers.
In Subsection \ref{stpl1} we give proofs for Theorem  \ref{nm3} about sufficient condition of singularity in the case $m+n=3$, which is rather simple.
In Subsection  \ref{stpl2} we give proofs for Theorem ~\ref{OptiOrdi} and Theorem ~\ref{optimn} which ensure upper bounds for exponents $\omega(\pmb{\xi})$ and $\hat{\omega}(\pmb{\xi})$ under various conditions of boundedness ratios.These proof do not rely on parametric geometry of numbers. They are based on analysis geometry of best approximation vector  and deal with well-chosen determinants. Here we should note that our proof of the second statement of Theorem \ref{optimn} (upper bound of $\hat{\omega}(\pmb{\xi})$  in the case $m=2$ when $\pmb{\xi}$ satisfies conditions (B) and (C) simultaneously) is related to a monotonicity lemma (Lemma \ref{monotonicity}) which is a variant of a lemma from \cite{Szeged}. In Appendix we give a proof of our Lemma \ref{monotonicity} and provide all the necessary adaptations.

 Section \ref{qq0} contains the proof of results related to application  main theorem from \cite{David}.
 Here we prove Theorem$^\star$~\ref{Thm-}, Theorem$^\star$~\ref{ThmOrd} and Theorem$^\star$~\ref{Main}. The structure of the proof is as follows. We effectively construct connected templates which satisfy the properties of boundedness of corresponding ratios, expresses in terms of parametric geometry of numbers, and then refer to \cite{David} to ensure the existence of matrices  $\pmb{\xi}$ with the similar properties.
 Templates for Theorem$^\star$~\ref{Thm-} are constructed in Subsection~\ref{beau}. Templates for Theorem$^\star$~\ref{ThmOrd} are constructed in Subsection~\ref{qq1}. Templates for Theorem$^\star$~\ref{Main}. are constructed in Subsection~\ref{qq3}.

\section{Parametric geometry of numbers}\label{PGN}
 
The \emph{parametric geometry of numbers} was developed by Schmidt and Summerer \cite{SS1,SS2} for $m=1$ or $n=1$ , following a question by Schmidt in \cite{SLum}. It was pushed further with a fundamental Theorem by Roy \cite{Roy}.  Das, Fishman, Simmons and Urba\'nski \cite{David} extends the setting to any $m,n\ge1$, with a quantitative aspect. Solan \cite{Solan} considers a weighted setting.

\subsection{Schmidt's observation on properties (B) and (C)}\label{123ss}
Here, we follow the notation and discussion from Schmidt in \cite{Wgene,S} and extend it to $\min(m,n)>1$.\\

Fix the dimension $m,n\ge1$ and $\bxi \in \bbR^{n \times m}$ { totally irrational}. For a real $q>0$, we consider the symmetric convex body
\begin{equation}\label{defconv}
\cC_\bxi (q) = \left\{ { \bx =  (\pmb{x},\pmb{y})^\top } \in \bbR^{m+n}  \mid \|\pmb{x}\| \le e^{nq}, \quad L_{\bxi}(\bx) \le e^{-mq} \right\}.
\end{equation}
The parametric geometry of numbers is the simultaneous study of the successive minima functions when the parameter $q$ tends to infinity. Namely, set 
\[h_{\bxi,d}(q) = \log \lambda_d (  \cC_\bxi(q))  \]
where $\lambda_d (K)$ is the d\textsuperscript{th} successive minima of the convex body $K$ with respect to the integer lattice $\bZ^{m+n}$. Minkowski's Second Convex Body Theorem ensures that 
\[ h_{\bxi,1}(q) + \cdots + h_{\bxi,m+n}(q) = O(1). \]
Denote by $\bh_\bxi = (h_{\bxi,1}, \ldots , h_{\bxi,m+n})$ the $(m+n)$-tuple of the logarithms of successive minima. Consider for any $\bx \in \bZ^{m+n}$ the gauge function
\[ \lambda_\bx \left(\cC_\bxi(q)\right) = \min \left(  \lambda >0 \mid \bx \in \lambda \cC_\bxi(q) \right)\]
and its logarithm 
\begin{equation}\label{defgauge}
h_\bx(q) := \log \left( \lambda_\bx \left(\cC_\bxi(q)\right)\right)  = \max \left(\log \| \bx\| -nq, \; \log L_{\bxi}(\bx) +mq \right).\end{equation}
Note that this function is first with slope $-n$, and then with slope $m$. Hence, since the joint graph of the components of $\bh_\bxi$ is contained in the union over $\bx \in \bZ^{m+n}$ of the graphs of $h_\bx$, the functions $h_{\bxi,d}(q)$ are piecewise linear with slope $-n$ and $m$, for $1\le d \le m+n$. Note that 
\begin{equation}\label{h1}
h_{\bxi,1}(q) = \min_{\bx \in \bZ^{m+n}}h_\bx(q)
\end{equation}
 and is exclusively composed of parts of $h_{\bx_i}(q)$, where $\bx_i$ is the sequence of best approximations to $\bxi$ (by definition of the latter). Consider now the sequence
\begin{equation}\label{defpiqi}
p_1 < q_1 < p_2 < q_2 < \cdots < p_i < q_i < p_{i+1} < q_{i+1} < \cdots
\end{equation}
such that $h_{\bxi,1}(q)$ has slope $-n$ on the intervals $[q_i, p_{i+1}]$ and slope $m$ on the intervals $[p_i,q_i]$. In \cite{Wgene}, the $q_i$'s are called \emph{nick numbers}. Since we supposed  $L_\bxi(\bx)>0$ for all $\bx \in \bZ^{m+n}$, this sequence is infinite. Up to reindexing, we may assume that $h_{\bxi,1}(q_i) = h_{\bx_i}(q_i) = h_{\bx_{i+1}}(q_i)$ and $h_{\bxi,1}(p_i) =  h_{\bx_i}(p_i)$  where $h_{\bx_i}$ changes of slope at $p_i$. See Figure \ref{fig1}. It follows that
\begin{eqnarray}
h_{\bx_i}(q_i)  &=&  m q_i + \log L_i  =  \log X_{i+1}  -n q_i    = h_{\bx_{i+1}}(q_i)\\
h_{\bx_i}(p_i) &=& \log X_i -n p_i= \log L_i + m p_i
\end{eqnarray}
from which we deduce
\begin{eqnarray}\label{pi}
(m+n) p_i  &=&  \log X_{i} - \log L_i,\\\label{q_i}
(m+n) q_i &=& \log X_{i+1} - \log L_i.
\end{eqnarray}

 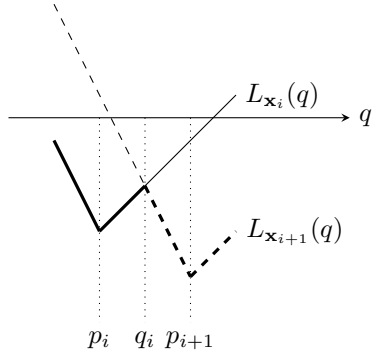
\begin{figure}[h!]
 \begin{center}
 \begin{tikzpicture}[scale=0.3]
 
   \draw[black, very thick] (0,0) --(2,-4) node [below,black] {$$} ;
   \draw[black, very thick] (2,-4) --(4,-2) node [below,black] {} ;
   \draw[black, thin] (4,-2) --(8,2) node [right,black] {$L_{\bx_i}(q)$} ;
   \draw[black, thin, dashed] (0,6) --(4,-2) ; 
   \draw[black, very thick, dashed] (4,-2) --(6,-6) node [below,black] {} ;
   \draw[black, very thick, dashed] (6,-6) --(8,-4)node [right,black] {$L_{\bx_{i+1}}(q)$} ;
   
   \draw[black, thin, dotted] (4,1) -- (4, -8) node [below,black] {$q_{i}$};
   \draw[black, thin, dotted] (2,1) -- (2, -8) node [below,black] {$p_{i}$};
   \draw[black, thin, dotted] (6,1) -- (6, -8) node [below,black] {$p_{i+1}$};
   
    \draw[black, thin , -stealth] (-2,1) --(13,1) node [right,black] {$q$} ;
;

 \end{tikzpicture}
 \end{center}
 \caption{$h_{\bx_i}(q)$ is figured plain, $h_{\bx_{i+1}}(q)$ is figured dashed, and $h_{\bxi,1}(q)$ is figured thick. Slopes are $-n$ and $m$, here chosen to be $-2$ and $1$.}\label{fig1}
 \end{figure}

Combining these observations yields
\begin{equation}\label{ssyq}
|q_i - p_i| = \frac{1}{m+n}\log \left( \frac{X_{i+1}}{X_i} \right) ,\,\,\,\,
|p_{i+1}- q_i | = \frac{1}{m+n} \log \left( \frac{L_{i}}{L_{i+1}} \right).
\end{equation}

It allows to interpret the conditions (B) and (C) in terms of the function $h_{\bxi,1}$. Namely, condition (B) is equivalent to

\begin{enumerate}[(b)]
\item{} $h_{\bxi,1}(q)$ has slope $m$ on intervals of bounded length, that is $|q_i - p_i| \ll 1$ for all $i\ge1$.
\end{enumerate}

Analogously  condition (C) is equivalent to 

\begin{enumerate}[(c)]
\item[(c)]{} $h_{\bxi,1}(q)$ has slope $-n$ on intervals of bounded length, that is  $|p_{i+1}- q_i | \ll 1$ for all $i\ge1$.
\end{enumerate}

Concerning bad approximability (condition (A)), observe that $h_{\bxi,1}(q)$ has a local minimum at $p_i$ where
\begin{equation}
h_{\bxi,1}(p_i) = \frac{1}{m+n} \log\left(  X_i^m L_i^n\right).
\end{equation}
Hence, (A) is equivalent to 
\begin{enumerate}[(a)]
\item $h_{\bxi,1}(q)$ remains at bounded distance from $0$.\\
\end{enumerate}
 
 Analogously, $\bxi$ is singular if $\limsup_{q \to \infty}  h_{\bxi,1}(q) = -\infty$.\\
 
 We can also use \eqref{pi}, \eqref{q_i} to evaluate the ordinary and uniform exponents of $\bxi$. Namely, equalities
\begin{equation}
\frac{h_{\bxi,1}(p_i)}{p_i} = \frac{m+n}{1- \frac{\log L_i}{\log X_i}} -n     \quad , \quad \frac{h_{\bxi,1}(q_i)}{q_i}  = \frac{m+n}{1- \frac{\log L_i}{\log X_{i+1}}} -n  
\end{equation}
lead to the formulae
\begin{equation}\label{LimPOmega}
\limsup_{q\to\infty}\frac{h_{\bxi,1}(q)}{q} = \frac{m-n \hat{\omega}(\bxi)}{1+ \hat{\omega}(\bxi)} \; \textrm{ and } \; \liminf_{q\to\infty}\frac{h_{\bxi,1}(q)}{q} = \frac{m-n{\omega}(\bxi)}{1+ {\omega}(\bxi)}.
\end{equation}
 
 \subsection{Templates and variational principle}\label{123sss} 
 We now define the family $\cT_{m,n}$ of balanced  $(m,n)$-templates. It was first introduced in \cite{David} up to a multiplying factor $mn$. It naturally generalizes the notions of systems introduced by Schmidt-Summer and Roy to $\min(m,n)>1$.
\begin{defn}\label{DefTemp}
We define a  (balanced) \,  $(m,n)$-template as a $(m+n)$-tuple of functions $\bP = (P_1, \ldots , P_{m+n})$ defined on an interval $I$ having
  \begin{enumerate}[(i)]
  \item{} $P_1(q) + \cdots + P_{m+n}(q) = 0$ for all $q \in I$.
  \item{}$ P_1(q) \le P_2(q) \le \cdots \le P_{m+n}(q)$ for all $q \in I$.
  \item{} $P_i$ is piecewise linear with slopes $\frac{-kn + lm}{k+l} \in [-n,m]$ where $k+l \ge 1$ and $0\le k \le m$ and $0 \le l \le n$, with finitely many pieces in any subinterval not containing $0$.
  \item{} For all $j = 1, \ldots , m+n$ and for every interval $I$ such that $P_j < P_{j+1}$ on $I$, the function $F_j = \sum_{i \le j} P_i$ is convex and piecewise linear on $I$  with slopes $-kn + lm$, where $k+l \ge 1$ and $0\le k \le m$ and $0 \le l \le n$.\\
  \end{enumerate}
\end{defn}
This family of functions  $\cT_{m,n}$ proves to approach precisely the successive minima functions $\bh_{\bxi}$ when $\bxi$ ranges through $\bbR^{n \times m}$. This is a fundamental Theorem of Roy \cite{Roy}
(case $\min(m,n)=1$), extended to $\min(m,n)>1$ by Das, Fishman, Simmons and Urba\'nski \cite{David} with a quantitative aspect. In our proofs of Theorem$^\star$~\ref{Thm-}, Theorem$^\star$~\ref{ThmOrd} and Theorem$^\star$~\ref{Main},  we construct relatively simple families of connected templates $ \bP$.
Then we use the main result from \cite{David}, recalled in Theorem \ref{VarPrinc} below to ensure the existence of matrices $\pmb{\xi}$  which correspond to this templates. That is the successive minima functions of $\pmb{\xi}$ is at bounded distance from elements of $ \bP$.

\begin{remark}\label{rmktpq}
The subfamily of templates with slopes restricted to the set $\{m,-n\}$ is dense in the set of templates. One can suppose this restriction without loss of generality. This allows for considering for such a template the sequence \eqref{defpiqi}. 
\end{remark}

Before we state this \emph{variational principle in parametric geometry of numbers}, we need to introduce some quantities and definition. 

For $q >0$, one can define the \emph{local contraction rate} $\delta(\bP,q)$ of a $(m,n)$-template $\bP$ at $q$. The definition can be found in \cite[Definition 4.5]{David}. We omit it here as it is unnecessarily technical. For our purpose, we only need the following :

\begin{equation}\label{Delta0}
\bP(q)= \boldsymbol{0}, \forall q \in [q_1,q_2]  \Rightarrow \delta(\bP,q) = mn, \forall q \in [q_1,q_2]
\end{equation}
Note that $mn$ is the maximal value of $\delta(\bP,q)$. Denote the average contraction rate on an interval $[q_1,q_2]$ by
\[ \delta([q_1,q_2]) = \frac{1}{q_2-q_1} \int_{q_1}^{q_2} \delta(\bP,q) \dd q.\]
The \emph{upper (resp. lower) average contraction rate} $\overline{\delta}(\bP)$ (resp. $\underline{\delta}(\bP)$) of a template $\bP$ is defined by
\[\underline{\delta}(\bP) = \liminf_{q\to\infty} \delta([q_1,q]), \]
\[\overline{\delta}(\bP) = \limsup_{q\to\infty} \delta([q_1,q]). \]

Given a $(m,n)$-template $\bP$, consider sets
\[\mathcal{M}(\bP) = \{ \bxi \in \bbR^{n \times m} \mid \Vert\bh_{\bxi}-\bP\Vert \textrm{ is bounded}\},\]
\[\mathcal{P}(\bP) = \{ \boldsymbol{Q} \in \cT_{m,n} \mid \Vert\boldsymbol{Q}-\bP\Vert  \textrm{ is bounded}\}. \] 

For example, $\mathcal{M}(\boldsymbol{0})$ is the set of all matrices $\bxi \in \bbR^{n \times m}$ such that its successive minima function $\bh_{\bxi}$ remains at bounded distance from $\boldsymbol{0}$. That is, each of its successive minima is bounded. Note that this finite bound can be arbitrarily large. We easily observe a characterization of the set of badly approximable matrices $\Bad_{n,m}$ in this context.

\begin{prop}\label{charbad}
For any dimension $n\ge1$ and $m\ge1$, 
\[ \Bad_{n,m} = \mathcal{M}(\boldsymbol{0}). \]
\end{prop} 

For our purpose, we reduce the \emph{variational principle} as follows (See \cite{David} for more).

\begin{maintheorem}[Das, Fishman, Simmons, Urba\'nski \cite{David}] \label{VarPrinc}
Let $\bP\in \cT_{m,n}$ be a template defined on $[Q_1, \infty)$. Furthermore,
\begin{equation}
\dim_H \mathcal{M}(\mathcal{P}(\bP))  = \sup_{\boldsymbol{Q}\in\mathcal{P}(\bP)} \underline{\delta}(\boldsymbol{Q}),\end{equation}
\begin{equation}
\dim_P \mathcal{M}(\mathcal{P}(\bP))  = \sup_{\boldsymbol{Q}\in\mathcal{P}(\bP)} \overline{\delta}(\boldsymbol{Q}).\end{equation}
\end{maintheorem}

{ In particular, note that the set $\mathcal{M}(\bP)$ is not empty whenever $\sup_{\boldsymbol{Q}\in\mathcal{P}(\bP)} \overline{\delta}(\boldsymbol{Q})>0$, but might be empty if the latter is $0$. \\}

The next proposition generalises the discussion by Schmidt in \cite[\S 5]{S}.
\begin{prop}\label{prop}
Let (k) be one of the properties (a), (b) or (c). Suppose that $\bP\in \cT_{m,n}$ is a $(m,n)$-template { satisfying (k), such that there exists a constant $\kappa >0$ so that each maximal interval of linearity of $\bP$ has length at least $\kappa$. Then, }any $(m,n)$-template in $\cP(\bP)$ also has property (k) and Theorem \ref{VarPrinc} provides a set of $\bxi\in\bbR^{n \times m}$ satisfying (k) of Hausdorff dimension at least $\underline{\delta}(\bP)$ and Packing dimension at least $\overline{\delta}(\bP)$.
\end{prop}
 Here, we abuse notation and say that $\bP$ satisfies (k) if its first component $P_1$ satisfies the same condition as $h_{\xi,1}$ in (k).\\
 
Hence, existence of $\bxi\in \bbR^{n \times m}$ claimed in Theorem$^\star$~\ref{Thm-} and Theorem$^\star$~\ref{Main} follows from the (explicit) construction of templates, where intervals of linearity are of length not shrinking to $0$, satisfying or not conditions (a), (b) or (c), and condition on the supremum and infimum limit of $P_1(q)/q$ stemming from condition on the ordinary or uniform exponent (see \eqref{LimPOmega}). Lower bounds for Hausdorff and packing dimension are induced for free by Theorem \ref{VarPrinc}.\\

The next Section \ref{noRoy} is devoted to the proof of all our claimed results that do not require the strength of Theorem \ref{VarPrinc}, namely Theorem \ref{Thm+}, Theorem \ref{nm3}, Theorem \ref{OptiOrdi} and Theorem \ref{optimn}. In Section \ref{Construction}, we construct $(m,n)$-templates proving Theorem$^\star$~\ref{Thm-}, Theorem$^\star$~\ref{ThmOrd} and Theorem$^\star$~\ref{Main}. We finish this section with the proof of Propositions \ref{charbad} and Proposition \ref{prop}.\\
 
 \subsection{Proof of Propositions \ref{charbad} and \ref{prop}} \label{123ssss}   
 
Proposition \ref{charbad} is given easily by the chain of equivalences
\begin{eqnarray*}
(A) \; \; \; \; \bxi\in \textrm{Bad}_{n,m} &\Leftrightarrow& (a) \; \; \; \; h_{\bxi,1} \textrm{ remains at bounded distance from } 0 \\
&\Leftrightarrow& \textrm{for all } 1\le j \le m+n, \; h_{\bxi,j} \textrm{ remains at bounded distance from } 0\\
&\Leftrightarrow& \bh_{\bxi} \in \mathcal{M}(\boldsymbol{0})
\end{eqnarray*} 

The second equivalence follows from Minkowski's Second Convex Body Theorem.\qed\\

Note that it is also equivalent to being badly approximable for approximation by rational subspaces, see Proposition 1.1 from \cite{Yorkies}, where the second equivalence is discussed in details.\\
 
The proof of Proposition \ref{prop} relies on the next Lemma~\ref{KeyLemma}. Before stating it, we need some extra notation. Consider the set $\cF (S)$ of piecewise linear functions $F =F(q)$ with a finite number of admissible slopes $S$, with finitely many pieces on any subinterval not containing $0$ or $\infty$. Denote the extremal slopes
 $$s_{\min} = \min_{s\in S} s \,\,\, \textrm{ and }\,\,  s_{\max} = \max_{s\in S} s.$$ 
Define
$$\mu = \min_{s_1, s_2 \in S: s_1\neq s_2}|s_1-s_2|.$$
 
Let $  F\in \cF (S)$, denote by $|\cdot|$ the length of an interval. We consider three conditions. 
 
 \vskip+0.3cm
  \noindent 
{\rm \textbf{ Condition U}}.  There exists $K>0$ and $\kappa >0$ with the following properties :
If $F$ is linear on the segment $J \subset [K, +\infty)$ but is not linear on any segment $ J' \supset J$ then $|J|>\kappa$.

  \vskip+0.3cm
 \noindent 
{\rm \textbf{ Condition V$_1$}}.
There exists $k_*\in \mathbb{N}$ and $\kappa >0$ with the following properties: 
    Split the interval $[1,+\infty)$ into the finite or infinite union
\[ [1,+\infty)= 
I_1 \cup J_1\cup I_2\cup J_2 \cup ... \cup  I_\nu\cup J_\nu \cup ...\, , \,\,\,\,
I_\nu = [u_\nu, v_\nu],\,\, J_\nu = [v_\nu, u_{\nu+1}],
\]
where $F$ has slope $s=s_{\max}$ on $I_i$ and slope $s <s_{\max}$ on $J_j$. 
Then for any $j$ there exists $j_*  \in \{ j, j+1,...,j+k_*\}$ such that 
 $|J_{j_*}|>\kappa$.
\vskip+0.3cm

 \noindent 
{\rm \textbf{ Condition V$_2$}}.
There exists $k_*\in \mathbb{N}$ and $\kappa >0$ with the following properties:
 Split the interval $[1,+\infty)$ into the finite or infinite union
\[ [1,+\infty)= 
I_1 \cup J_1\cup I_2\cup J_2 \cup ... \cup  I_\nu\cup J_\nu \cup ...\, , \,\,\,\,
I_\nu = [u_\nu, v_\nu],\,\, J_\nu = [v_\nu, u_{\nu+1}],
\]
where $F$ has slope $s=s_{\min}$ on $I_i$ and slope $s>s_{\min}$ on $J_j$. 
Then for any $j$, there exists $j_*  \in \{ j, j+1,...,j+k_*\}$ such that 
 $|J_{j_*}|>\kappa$.\\

 \begin{lemma}\label{KeyLemma}
Consider any functions $F,G \in \cF(S)$ such that $ \sup_q|F(q) -G(q)| \le c$.
\vskip+0.3cm
  \noindent
  {\rm (1)}
 For an arbitrary function $f(q)$ one has 
$$|f(q) - F(q)| \ll1, \quad \forall q>0 \quad  \Leftrightarrow \quad |f(q) - G(q)| \ll1 , \quad \forall q>0 .$$

\vskip+0.3cm
  \noindent
  {\rm (2)}
  Assume that $F\in \cF$ satisfies  condition
{\rm \textbf{ Condition U}} or {\rm \textbf{ Condition V$_1$}}. Then
 if $F$ has slope $s_{\max}$ only on intervals of length bounded by $\ell$, then $G$ has slope $s_{\max}$ only on intervals of length
 bounded by a constant depending on $\ell, \mu, \kappa, k_*$ and $c$.
 
\vskip+0.3cm
  \noindent
  {\rm (3)} 
    Assume that $F\in \cF$ satisfies  condition
{\rm \textbf{ Condition U}} or {\rm \textbf{ Condition V$_2$}}. Then
  if $F$ has slope $s_{\min}$ only on intervals of length bounded by $\ell$, then  $G$ has slope $ s_{\min}$ only on intervals of length bounded by a constant depending on $\ell, \mu, \kappa, k_*$ and $c$.
 
 \end{lemma}
 We must supply Lemma ~\ref{KeyLemma} with an important remark.
   \begin{remark}\label{rema11}
Consider a family $ \cF$ with  $s_{\max} = m $ and $s_{\min}=-n$. Then $\{ -n,m\} \subset S$. For any matrix $\bxi$, its successive minima function $h_{\bxi,1} $ satisfies  both  {\rm \textbf{ Conditions V$_1$,  V$_2$}}.
  Indeed,  $h_{\bxi,1}\in  \cF$ have only two changing  slopes $-n$ and $m$.
  For the length of the intervals of linearity for  the function $h_{\bxi,1}$, we have formulas 
 (\ref{ssyq}) which express the lengths in terms of ratios of best approximations 
 $\frac{X_{i+1}}{X_i}$ and $ \frac{L_{i+1}}{L_i}$.
  It is well-known (see \cite{MMJ,Che,NM}) that sequences of values $X_i$ and $L_i$ have exponential growth and decay, that is 
 $$
 X_{i+k_1} \ge 2 X_i,\,\,\,\,
 L_{i+k_2} \le \frac{L_i}{2},\,\,\,\, \forall \, i\in\bN
 $$
 for fixed positive integers $k_1, k_2$. So \textbf{Condition V$_1$} and \textbf{Condition V$_2$} are satisfied
 with $k_* = k_1, \kappa = \frac{\log 2}{k_1(m+n)}$ {or}  $k_*=k_2, \kappa = \frac{\log 2}{k_2(m+n)}$.\\
\end{remark}

{
\begin{proof}[Proof of Proposition ~\ref{prop} ]
Since $\bP$ is assumed to satisfy \textbf{ Condition U}, when (k) is respectively (a), (b) or (c) we apply respectively (1), (2) or (3) from Lemma ~\ref{KeyLemma} with $f(q):=0$, $F:=\bP$ and $G:=h_{\bxi}$ for any $\bxi\in \cM(\bP)$. The information about Hausdorff and Packing dimensions follows directly from
\begin{equation*}
\dim_H \mathcal{M}(\mathcal{P}(\bP))  = \sup_{\boldsymbol{Q}\in\mathcal{P}(\bP)} \underline{\delta}(\boldsymbol{Q}) \ge \underline{\delta}(\boldsymbol{P}),\end{equation*}
\begin{equation*}
\dim_P \mathcal{M}(\mathcal{P}(\bP))  = \sup_{\boldsymbol{Q}\in\mathcal{P}(\bP)} \overline{\delta}(\boldsymbol{Q}) \ge  \overline{\delta}(\boldsymbol{P}).\end{equation*}
\end{proof}}

\begin{proof}[Proof of Lemma ~\ref{KeyLemma} ]

Statement (1) follows directly from triangle inequality.\\

We prove  Statement (2).
 Suppose that $G$ has slope $s=s_{\max}$ on arbitrarily large intervals. Consider an arbitrarily large (but finite) interval $I=[p,q]$ such that $G$ has slope $s=s_{\max}$. By maximality of $s_{\max}$ on $I$, the slopes of the piecewise linear function $ G - F$ are everywhere non-negative. Split the interval $I$ into the disjoint union
\[ I = \bigcup_{i=1}^{i_0}I_i \cup \bigcup_{j=1}^{j_0}J_j   \]
where $F$ has slope $s=s_{\max}$ on $I_i$ and slope $s<s_{\max}$ on $J_j$. 

By hypothesis,
  $G - F$ has slope $s=0$ on $I_i$ and slope $s\ge \mu$ on $J_j$. Note that $|i_0-j_0|\le 1$.  It follows
  from the conditions of the lemma  that
$$
2c \ge {(G(q) - F(q)) - ( G(p) - F(p))} \ge \mu \sum_{j=1}^{j_0} |J_j| \ge  
\begin{cases}
\mu \kappa (j_0-2), \,\,\,\text{under  {\rm \textbf{ Condition U}}},
\cr
\mu\kappa \left[\frac{j_0-2}{k_*}\right],\,\,\,\text{under  {\rm \textbf{ Condition V$_1$}}},
 \end{cases} 
  $$
  Here, we exclude the very left and the very right intervals $J_\nu$ from the consideration because they may not be the complete intervals $J_\nu$ from  {\rm \textbf{Conditions U, V$_1$}}. So  
$$i_0, j_0 \le  W=
\begin{cases}
\frac{2c }{\mu\kappa}+3 ,  \,\,\,\,\,\,\,\,\,\,\,\,\,\,\,\,\,\,\,\,\,\,\,\,\,\,\,\,\,\,\,\,\,\,\,\,\,\,\,\,\,\,\,\,\,\,\,\,\,\,\,\,\,\,
\text{under  {\rm \textbf{ Condition U}}},
\cr
\max \left(j_*, \left(
\frac{2c }{\mu\kappa} +1\right) k_* \right)+3 ,\,\,\,\text{under  {\rm \textbf{ Condition V$_1$}}}.
 \end{cases} 
$$
 From $|I| = \sum_{i=1}^{i_0} |I_i| +\sum_{j=1}^{j_0} |J_j|$ we deduce that there exists $1\le i \le i_0$ such that
\[ |I_i| \ge
\frac{ |I| - \frac{2c}{\mu}}{W}
. \]
If $|I|$ was chosen large enough, it contradicts the upper  bound $\ell$ for  $|I_i|$. This proves (2).\\

Statement (3) is proved with analogous arguments. 
\end{proof}

 \subsection{Proof of Theorem \ref{normindep}}\label{Indepnorm}
 Lemma \ref{KeyLemma} allows to prove Theorem \ref{normindep}.
Here we explain how the boundedness of ratios for best approximations in sup-norm ensure the boundedness of ratios for the best approximations in arbitrary norm. The proof of the converse statement is similar.
 
Consider two norms $\|\cdot\|_{1}$ and $\|\cdot\|_{2}$, and define the modified convex body
 $$\cC'_{\bxi}(q) :=  \left\{  { \bx =  (\pmb{x},\pmb{y})^\top } \in \bbR^{m+n}  \mid \|\bx\|_1 \le e^{nq}, \quad \|\pmb{\xi}\pmb{x}-\pmb{y}\|_2  \le e^{-mq} \right\}.$$ 
Denote by $\bh_{\bxi}'$, $h_{\bx}'$ its associated minima and gauge functions. 
 If $\|\cdot\|_{1}$ and $\|\cdot\|_{2}$ are the sup norms, this is just the convex body we defined in \eqref{defconv}.
 By equivalence of the norms in $\bbR^{m+n}$ and \eqref{defgauge}, we obtain that for any $\bx$ the gauge functions $h_{\bx}'$ and $h_{\bx}$ remain at bounded distance. Hence by observation \eqref{h1} so do the successive minima functions $h_{\bxi,1}'$ and $h_{\bxi,1}$.
 Because of Remark \ref{rema11}
 we can apply Lemma \ref{KeyLemma}.
 and
 independence on the norm  follows.\qed
 
 The sequence of best approximations of course depends on the norm. Note that Roy in \cite{Roy} uses the Euclid{e}an norm, while Schmidt and Summerer \cite{SS1,SS2} use the sup norm.

 \section{Proofs not using the variational principle}\label{noRoy}
 In this section, we prove all our claimed results that do not require the strength of Theorem \ref{VarPrinc}, namely Theorem \ref{Thm+}, Theorem \ref{nm3}, Theorem \ref{OptiOrdi} and Theorem \ref{optimn}.
 
 \subsection{Proof of Theorem \ref{Thm+}} 
 \label{stpl}
 
 We provide here a new alternative proof in the language of the parametric geometry of numbers. See \cite{AM} for a more classical setting.\\
  We prove the implications of Theorem \ref{Thm+} in the same order. Remember that $\bxi\in \bbR^{n \times m}$ is totally irrational. 

\begin{itemize}
\item \emph{Suppose that (a) holds}, that is $|h_{\bxi,1}(q)|< \ell$ for some constant $\ell$ and all $q>0$. For $s\in\{ -n,m\}$ suppose that $h_{\bxi,1}$ has slope $s$ on the interval $(u,v)$, that is $$h_{\bxi,1}(v)=h_{\bxi,1}(u) + s(v-u).$$ Then 
 \[ |v-u|  = \frac{ \left|  h_{\bxi,1}(v) - h_{\bxi,1}(u) \right|}{|s|} \leq \frac{2\ell}{|s|}.\]
In particular, $h_{\bxi,1}$ can have slope $s=m$ on intervals of length at most $\frac{2\ell}{m}$ and slope $s=-n$ on intervals of length at most $\frac{2\ell}{n}$. We deduce that
 (a) implies both (b) and (c).\\ 
 
\begin{wrapfigure}{r}{0.5\textwidth}

\end{wrapfigure}

\item \emph{Suppose $m=1$ and $(b)$ holds}. By Schmidt and Summerer \cite{SS1,SS2} and Lemma \ref{KeyLemma} (2), there exists a template $P$ with slopes restricted to $\{1,-n\}$, at bounded distance to $h_\xi$, satisfying (b). That is (see Remark \ref{rmktpq}) there exists $\ell>0$ such that $|q_i-p_i| <\ell$. By Lemma \ref{KeyLemma} (1), it is enough to show that this $P$ satisfies $(a)$. For this, we uniformly bound $P_1(p_i) = \min_{[q_{i-1},q_i]}P_1(q)$. Consider $i\in\mathbb{N}$, and its associated interval $[q_{i-1},q_i]$. Recall that $P_{1+n}$ may have only slopes $-n$ and $1$. We consider two cases.\\

\textbf{ Case 1:} Suppose first that $P_{1+n}$ has constant slope $-n$ on a subinterval $[u_i,v_i]\subset [q_{i-1},q_i]$ where $u_i$ is minimal.
Since for each $q$, there is only one component $P_j$ with slope $-n$, we deduce 
$$p_i\le u_i \quad \textrm{ and } \quad q_i \ge u_i + \frac{P_{1+n}(u_i) - P_1(u_i)}{n+1}.$$
{ The second inequality follows from the observation that the line of slope $-n$ passing through $(u_i,P_{1+n}(u_i))$ lies below the line of slope $-n$ passing through $(q_i,P_1(q_i))$. See Figure ~\ref{Fig:Race}. Indeed, at each $q$ there is only one $P_i$ with negative slope. Condition (iv) on slopes in the Definition \ref{DefTemp} of a template says that at the turning point $q$ such that this unique negative slope changes from $P_i$ to $P_j$, we have either $j>i$ and the new line of negative slope is above the previous one, or $P_i(q)=P_j(q)$, meaning that we stay on the same line of negative slope. Note that if we have $m>1$ negative slopes, we can only infer that the lowest line with negative slope at some $q$ is below the lowest line with negative slope at some $p>q$, and our observation is not valid anymore.\\}
 By minimality of $u_i$, $P_{1+n} -P_1$ is constant on $[p_i,u_i]$, so $$P_{1+n}(p_i) - P_1(p_i)=P_{1+n}(u_i) - P_1(u_i),$$ and
$$u_i \le \nu_i := q_i - \frac{P_{1+n}(p_i) - P_1(p_i)}{n+1}.$$
Hence, recalling $P_{1+n} \ge 0 \ge P_1$
$$|P_1(p_i)| \le |P_{1+n}(p_i) - P_1(p_i)| \le (n+1)(q_i-u_i) \le \ell(n+1).$$
This proves the first case.\\

\textbf{Case 2:} Suppose now that $P_{1+n}$ has constant slope $1$ on $[q_{i-1},q_i]$.\\
Consider the index 
$$j:=\min\{j>i \mid P_{1+n} \textrm{ has slope $-n$ on a subinterval of } [q_{j-1},q_j] \}.$$ 
Since $\xi$ is supposed to be totally irrational, $j$ is well defined. By construction, 
$$P_{1+n}(p_i) = P_{1+n}(p_j) - (p_j-p_i) \le P_{1+n}(p_j)$$
and the index $j$ falls into the first case so that $|P_1(p_j)| \le \ell(n+1)$. Recall that $\frac{-P_1}{n} \le P_{1+n} \le -nP_1$, we deduce
$$\left| \min_{[q_{i-1},q_i]}P_1(q)\right| = |P_1(p_i)| \le n P_{1+n}(p_i) \le n P_{1+n}(p_j) \le n^2 |P_1(p_j)| \le \ell n^2(n+1).$$
This proves the second case.\\

Thus $(a)$ holds for $P$ and by Lemma \ref{KeyLemma} (1), also holds for $\bh_\bxi$ and $\bxi\in Bad_{1,n}$.\\

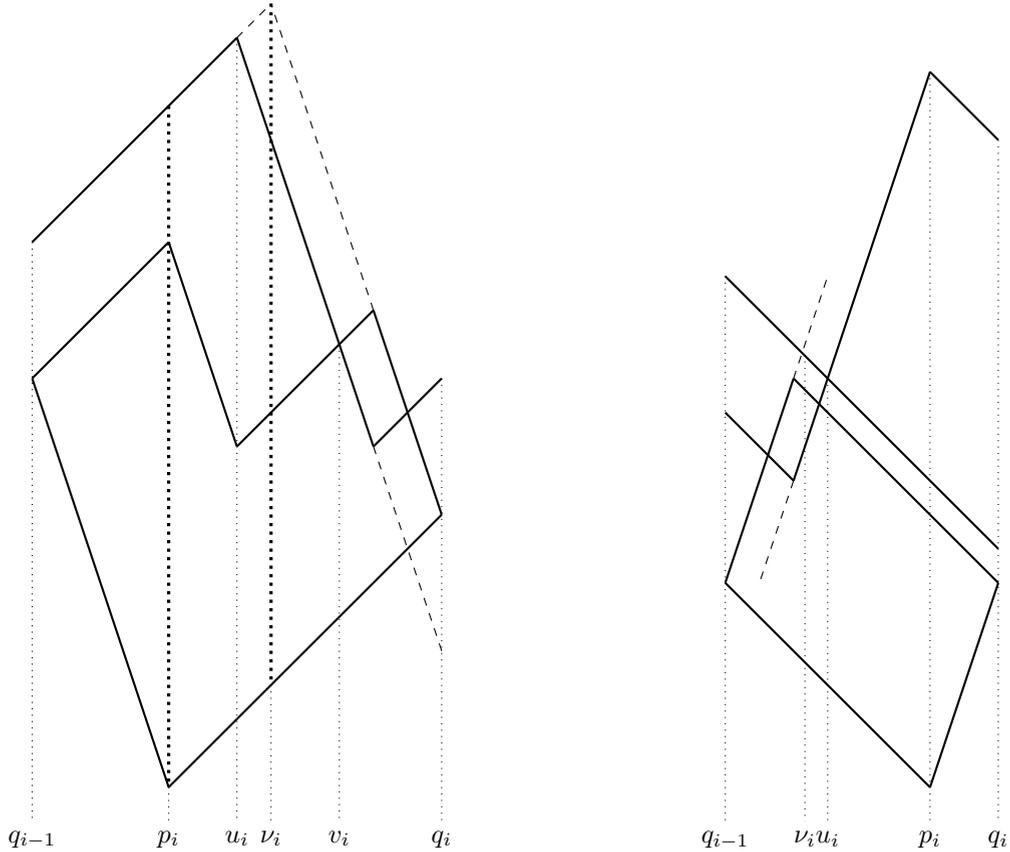
\begin{figure}[h]
\begin{tabular}{llr}
\begin{tikzpicture}[scale=0.45]
  
    \draw[thick] (0,0) -- (4,-12);
     \draw[thick] (0,0) -- (4,4);
     \draw[thick] (0,4) -- (4,8);
     
     \draw[thick] (4,-12) -- (6,-10);
    \draw[thick] (4,4) -- (6,-2);
    \draw[thick] (4,8) -- (6,10);
     
    \draw[thick] (6,-10) -- (10,-6);
    \draw[thick] (6,-2) -- (10,2);
   \draw[thick] (6,10) -- (10,-2);
     
   \draw[thick] (10,-6) -- (12,-4);
    \draw[thick] (10,2) -- (12,-4);
   \draw[thick] (10,-2) -- (12,0);
     
     \draw[dashed] (10,2) -- (7,11);
     \draw[dashed] (6,10) -- (7,11);
     \draw[dashed] (10,-2) -- (12,-8);

    \draw[dotted,very thick] (4,8)--(4,-12);
    \draw[dotted,very thick] (7,11)--(7,-9);
    
    \draw[dotted] (0,4) -- (0,-13) node[below]  {$q_{i-1}$};
    \draw[dotted] (4,-12) -- (4,-13) node[below]  {$p_{i}$};
    \draw[dotted] (12,0) -- (12,-13) node[below]  {$q_{i}$};
    
    \draw[dotted] (9,1) -- (9,-13) node[below]  {$v_{i}$};
    \draw[dotted] (6,10) -- (6,-13) node[below]  {$u_{i}$};
    \draw[dotted] (7,-9) -- (7,-13) node[below]  {$\nu_{i}$};
  
\end{tikzpicture}
& .\hspace{2cm} .&
\begin{tikzpicture}[scale=0.45]
  
    \draw[thick] (0,0) -- (2,-2);
    \draw[thick] (0,0) -- (2,6);
    \draw[thick] (0,5) -- (2,3);
    \draw[thick] (0,9) -- (2,7);
    
    \draw[thick] (2,-2) -- (6,-6);
    \draw[thick] (2,6) -- (6,2);
    \draw[thick] (2,3) -- (6,15);
    \draw[thick] (2,7) -- (6,3);

    \draw[thick] (6,-6) -- (8,0);
    \draw[thick] (6,2) -- (8,0);
    \draw[thick] (6,15) -- (8,13);
    \draw[thick] (6,3) -- (8,1);
    
    \draw[dashed] (2,6) -- (3,9);
     \draw[dashed] (2,3) -- (1,0);
      
    \draw[dotted] (0,9) -- (0,-7) node[below]  {$q_{i-1}$};
    \draw[dotted] (8,13) -- (8,-7) node[below]  {$q_{i}$};
    \draw[dotted] (6,15) -- (6,-7) node[below]  {$p_{i}$};

    \draw[dotted] (3,6) -- (3,-7) node[below]  {$u_{i}$};
    \draw[dotted] (7/3,6.5) -- (7/3,-7) node[below]  {$\nu_{i}$};
    
  \end{tikzpicture}
\end{tabular}
\caption{Case (b) left and case (c) right.}
\label{Fig:Race}
\end{figure}

 \item \emph{Suppose $n=1$ and $(c)$ holds.} The proof is analogous to previous claim. Again by Schmidt and Summerer \cite{SS1,SS2} and Lemma \ref{KeyLemma} we can instead of $\bh_\bxi$ consider a template $P$ with slopes restricted to $\{1,-n\}$ satisfying $(c)$. That is (see Remark \ref{rmktpq}), there exists a constant $\ell>0$ such that $|p_i-q_{i-1}|<\ell$.\\
 The template has $m$ components with slope $-1$ and one component with slope $m$. Consider $i\in\mathbb{N}$, and its associated interval $[q_{i-1},q_i]$. There are two cases :\\

\textbf{ Case 1:} Suppose first that $P_{m+1}$ has constant slope $m$ on some interval $[u_i,v_i]\subset [q_{i-1},q_i]$. Since there is only one component $P_j$ with slope $m$, 
$$ v_i \le p_i \quad \textrm{ and } \quad u_i \ge q_{i-1} + \frac{P_{m+1}(q_{i-1}) - P_1(q_{i-1})}{m+1} := \nu_i.$$
{The second inequality is provided by the observation that the line of slope $m$ passing through $(q_{i-1},P_1(q_{i-1}))$ lies above the line of slope $m$ passing through $(u_i,P_{m+1}(u_i))$. See Figure ~\ref{Fig:Race}. Indeed, condition (iv) on slopes in the Definition \ref{DefTemp} of a template ensures that for any turning point $q$ such that the unique positive slope changes from $P_i$ to $P_j$, we have either $j<i$ and the new line of positive slope is below the previous one or $P_i(q)=P_j(q)$ and we stay on the same line of positive slope. Note that if $n>1$, we can only infer that the highest line with positive slope at some $q$ is above the highest line with positive slope at some $p>q$, and our observation is not valid anymore.\\
Recalling $P_{m+1}(q) \ge  0\ge P_1(q)$,}
$$|P_1(q_{i-1})| \le |P_{m+1}(q_{i-1}) - P_1(q_{i-1})| \le (u_i -q_{i-1})(m+1) \le \ell(m+1).$$
Hence,
$$\left| \min_{[q_{i-1},q_i]}P_1(q)\right| = |P_1(p_i)| \le |P_1(q_{i-1}) -(p_i-q_{i-1})| \le \ell(m+2).$$

\textbf{ Case 2:} Suppose now that $P_{m+1}$ has constant slope $-1$ on $[q_{i-1},q_i]$. Consider the index
$$j:= \max\{ j<i \mid P_{m+1} \textrm{ has slope $m$ on a subinterval of  } [q_{j-1},q_j] \}.$$
By construction, $P_{m+1}(p_i) \le P_{m+1}(p_j)$ and 
$$\left| \min_{[q_{i-1},q_i]}P_1(q)\right| = |P_1(p_i)| \le mP_{m+1}(p_j) \le m^2 |P_1(p_j)| \le \ell(m+2)m^2.$$
This proves that property $(a)$ holds for $P$, and by Lemma \ref{KeyLemma} (1) also holds for $\bh_\bxi$ and $\bxi\in \Bad_{m,1}$.

 \end{itemize}
 
 \subsection{Proof of Theorem \ref{nm3}} \label{stpl1}
 The proof of Theorem \ref{nm3} relies on the following key observation.
  
  \begin{lemma}\label{lem}
  Condition (b) or (c) implies that $h_{\bxi,1} (q)$ and $h_{\bxi,2}(q)$ remain at bounded distance, that is
  $h_{\bxi,2}(q) - h_{\bxi,1}(q) \ll1$
  for all $q>0$.
  \end{lemma}
  
  Indeed, suppose $m+n =3$ and conditions of Theorem \ref{nm3} (Statement 1 or Statement 2) are satisfied. Since $h_{\bxi,2}(q) = h_{\bxi,3}(q)$ for certain arbitrarily large values of $q=q_0$ (See \cite[Corollary 2.2]{SS1} or \cite[Theorem 9.2]{Wgene}), at such $q_0$, we have $h_{\bxi,2}(q_0) = h_{\bxi,3}(q_0) = h_{\bxi,1}(q_0) + O(1)$. Remember that Minkowski's second convex body theorem provides $h_{\bxi,1}(q)+ h_{\bxi,2}(q) + h_{\bxi,3}(q) =  O(1)$. Combining this observations leads to  $|h_{\bxi,1}(q_0)| = O(1)$, meaning that
  $\bxi\in \bbR^{n \times m}$ is not singular. In particular, $$\limsup_{q\to\infty} \frac{h_{\bxi,1}(q)}{q}  = 0 = \frac{m-n\hat{\omega}(\bxi) }{1+\hat{\omega}(\bxi)}, \; \textrm{ i.e. } \;  \hat{\omega}(\bxi) = \frac{m}{n}.$$
 
  \begin{proof}[Proof of Lemma \ref{lem}]
 { At $q_i$, $h_{\bxi,1}(q_i) = h_{\bxi,2}(q_i)$ and the difference $|h_{\bxi,1} - h_{\bxi,2}|$ has slope at most $m+n$. In particular, on the interval $[q_i,q_{i+1}]$,
  \[ \max_{[q_i,q_{i+1}]}|h_{\bxi,1}(q) - h_{\bxi,2}(q)| \le \min \{(m+n) |p_{i+1}-q_i |, (m+n) |q_{i+1}-p_{i+1} |\}   \]}
  and both (b) or (c) implies that $h_{\bxi,2} - h_{\bxi,1} \ll1$ with an explicit bound depending on dimension.
  \end{proof}

 \subsection{Proof of upper bounds for admissible exponents} \label{stpl2}
 We consider first upper bounds for the admissible ordinary exponent (Theorem ~\ref{OptiOrdi}), and then upper bounds for the admissible uniform exponent (Theorem ~\ref{optimn}).

 
 \subsubsection{Proof of Theorem~\ref{OptiOrdi} } \label{OptiOrdi1} We prove the three statements independently.
 
 \begin{itemize} 
 \item{Proof of Statement 1.}\\
  For $m=1$, consider for $i\in \bN$ the two-dimensional lattices $\Lambda_i = \langle  \bx_{i-1},\bx_i\rangle_\mathbb{Z}$. It is well known that in the case of simultaneous approximation ($m = 1$) one has
\begin{equation}\label{coovol}
L_{i-1} X_i \asymp {\rm covol}\,  \Lambda_i \to_{ i\to \infty} \infty.
\end{equation}
See for example \cite[Theorem 17]{Mo1}, or \cite{Mo2}, or the original statement from \cite[Satz 9]{Jar}. Now (C) means $L_{i-1}\asymp L_i$, leading to $L_i X_i \to_{ i\to \infty}  \infty$. That implies $\omega(\bxi) \le 1$.\\

\item{Proof of Statement 2.}\\
For $m=2$, we distinguish two cases.

\begin{itemize}
\item  In the \textbf{first case}, all the best approximation vectors  $\bx_{j}$ 
{ eventually}
lie in a two-dimensional linear subspace $\mathcal{B}\subset \mathbb{R}^d, d={2+n}$. 
{ As it was shown in  \cite{Mo2} this
  never happens for \emph{completely irrational} matrices $\pmb{\xi}$ (in the sense of the definition from \cite{Mo2};
  note that the notion of complete irrationality differs from  four notion of total irrationality which we use through this paper and which was introduced in  Section \ref{beginn}). See  Section 3 from \cite{Mo2},  as well as Section 3.2, Theorem 3.5  and  further discussion from \cite{Neck}.  

Consider two-dimensional subspace $\frak{L}$ defined in (\ref{frakel}).
As the matrix $\xi$ is totally irrational, rational two-dimensional   subspace $\mathcal{B}$ cannot coincide with the two-dimensional subspace $\frak{L}$.}
Then the intersection { $\ell=\frak{L}\cap \mathcal{B}$} is a one-dimensional linear subspace. So this case is completely similar to the case of simultaneous approximation ($m=1$) considered in Statement 1 above and $\omega(\bxi) \le 1$.
{
Indeed, the intersection $\Lambda =\mathbb{Z}^d\cap \mathcal{B}$ is a two-dimensional lattice.
As now $\frak{L}+\mathcal{B}$ is a three-dimensional linear subspace 
and eventually $\bx_i \in \mathcal{B} \subset \frak{L}+\mathcal{B}$,
we see that 
$$L_i\asymp {\rm dist} (\bx_{i},\ell), \,\,\,\,\, i\to \infty,$$
 where ${\rm dist} (\cdot,\cdot
)$ denotes the Euclidean distance.
Meanwhile,  lattice points  $\bx_{i} \in \Lambda \subset \mathcal{B}$ are "best approximations in induced norm" to the line $\ell$ and in particular the two-dimensional $0$-symmetric convex set
$$\{\bx =( \pmb{x},\pmb{y}):\,\, |\pmb{x}|\le X_i,\,\, L_{\pmb{\xi}}(\pmb{x})\le L_{i-1}\}
\cap \mathcal{B} \subset \mathcal{B}$$
contains no other points of $\Lambda$ than $\pmb{0}, \pm \bx_{i-1},\pm \bx_{i}$. 
Together with condition (C) this means that 
$$
L_{i} X_i\asymp
L_{i-1} X_i\asymp {\rm dist} (\bx_{i-1},\ell)  \cdot X_i \asymp {\rm covol}\,  \Lambda.
$$
So $L_{i} X_i$ is bounded away from zero and infinity and $\omega(\bxi) =1$.
}

\item
In the \textbf{second case}, there exist infinitely many triples of successive vectors   
\begin{eqnarray*}
\bx_{i-1} &=& (x_{1,i-1},x_{2,i-1},{y}_{1,i-1},...,{y}_{n,i-1})^\top,\\
\bx_{i} &=& (x_{1,i},{x_{2,i}}, {y}_{1,i},...,{y}_{n,i})^\top,\\
\bx_{i+1} &=& (x_{1,i+1},x_{2,i+1},{y}_{1,i+1},...,{y}_{n,i+1})^\top
\end{eqnarray*}

 which are independent.
We consider two three-dimensional vectors
\begin{equation}\label{pro1}
\frak{x}_1 = (x_{1,i-1},x_{1,i}, x_{1,i+1}),
\,\,\,
\frak{x}_2 = (x_{2,i-1},x_{2,i}, x_{2,i+1})
\in \mathbb{Z}^3.
\end{equation}
{It} may happen that these two vectors are dependent (proportional) or not.

If the vectors \eqref{pro1} are not proportional to each other,  there exists $1\le j \le n$ such that all the three three-dimensional vectors
$$\frak{x}_1, \frak{x}_2\,\,\, \text{and}\,\,\, \frak{y} = (y_{j,i-1},y_{j,i}, y_{j,i+1})
$$
are independent (because of independence of the vectors $\bx_{i-1},\bx_{i}, \bx_{i+1}$).
In other words,
$$
1\le
\left| {\rm det}
\left(\begin{array}{ccc}
x_{1,i-1}&x_{1,i}& x_{1,i+1}\cr
x_{2,i-1}&x_{2,i}& x_{2,i+1}\cr
y_{j,i-1}&y_{j,i}& y_{j,i+1}
\end{array}
\right)\right|=
$$
$$
=
\left| {\rm det}
\left(\begin{array}{ccc}
x_{1,i-1}&x_{1,i}& x_{1,i+1}\cr
x_{2,i-1}&x_{2,i}& x_{2,i+1}\cr
y_{j,i-1}-\xi_{j,1}x_{1,i-1} - \xi_{j,2}x_{2,i-1}&
y_{j,i}-\xi_{j,1}x_{1,i} - \xi_{j,2}x_{2,i}& 
y_{j,i+1}-\xi_{j,1}x_{1,i+1} - \xi_{j,2}x_{2,i+1}
\end{array}
\right)\right|
$$
$$
\ll
X_{i} X_{i+1}L_{i-1}.
$$
But $ X_i\le X_{i+1}$ , while $ L_{i-1}\ll L_{i+1}$ by condition (C). So
$$
L_{i+1} X_{i+1}^2 \gg 1
$$  infinitely often, and this means $ \omega(\bxi) \le 2$.

Now we deal with the case when the three-dimensional vectors $\frak{x}_1, \frak{x}_2$ are proportional. As three vectors $\bx_{i-1},\bx_{i}, \bx_{i+1}$ are independent,
there exist two three-dimensional vectors 
$$
\frak{y}' = (y_{j,i-1},y_{j,i}, y_{j,i+1})\,\,\,
\text{and}
\,\,\,
\frak{y}'' = (y_{l,i-1},y_{l,i}, y_{l,i+1}),\,\,\,
1\le j\neq l \le n,
$$
such that both triples 
$
\frak{x}_1,
\frak{y}' ,
\frak{y}'' 
$
and
$
\frak{x}_2,
\frak{y}' ,
\frak{y}'' 
$
consist of independent vectors.
Then, similarly to the previous case we have the bound
$$
1\le
\left| {\rm det}
\left(\begin{array}{ccc}
x_{1,i-1}&x_{1,i}& x_{1,i+1}\cr
y_{j,i-1}&y_{j,i}& y_{j,i+1}\cr
y_{l,i-1}&y_{l,i}& y_{l,i+1}
\end{array}
\right)\right|=
$$
$$
=
\left| {\rm det}
\left(\begin{array}{ccc}
x_{1,i-1}&x_{1,i}& x_{1,i+1}\cr
y_{j,i-1}-\xi_{j,1}x_{1,i-1} - \xi_{j,2}x_{2,i-1}&
y_{j,i}-\xi_{j,1}x_{1,i} - \xi_{j,2}x_{2,i}& 
y_{j,i+1}-\xi_{j,1}x_{1,i+1} - \xi_{j,2}x_{2,i+1}\cr
y_{l,i-1}-\xi_{j,1}x_{1,i-1} - \xi_{j,2}x_{2,i-1}&
y_{l,i}-\xi_{j,1}x_{1,i} - \xi_{j,2}x_{2,i}& 
y_{l,i+1}-\xi_{j,1}x_{1,i+1} - \xi_{j,2}x_{2,i+1}
\end{array}
\right)\right|
$$
$$
\ll X_{i+1}L_iL_{i-1}\ll X_{i+1} L_{i+1}^2.
$$
If this case occurs infinitely often we have $\omega(\bxi) \le \frac{1}{2}  <2$. Everything is proven.\\
\end{itemize}


\item{Proof of Statement 3.}\\
 The proof for $m\ge 3$ relies on the fact that norms of best approximations have exponential growth, as first observed by Voronoi. We use the following result of Bugeaud and Laurent~\cite[Lemma 1]{MMJ} : for any $\bxi\in\bbR^{n \times m}$ we have
\begin{equation}  \label{eq:buglau}
X_{k+3^{m+n}-1}\ge 2X_k.
\end{equation}

{
Note that this inequality holds for all possible values of $k$. We consider only totally irrational $\pmb{\xi}$, thus (\ref{eq:buglau}) holds for all positive integers $k$.
In full generality, (\ref{eq:buglau}) holds only for $k$ such that the best approximation $\bx_{k+3^{m+n}-1}$  is defined.\\}

Given $\bxi\in\mathbb{R}^{n\times m}$, by property (C) there is some $c_1=c_1(\bxi)>0$ for which we have
\[L_k = L_1 \cdot \frac{ L_2 }{L_1}\cdot \frac{ L_3}{L_2}\cdot \cdots \cdot \frac{ L_k}{ L_{k-1}  } \gg c_1^{k},	\]
or equivalently
\[-\frac{  \log L_k}{k} \le -\log c_1+o(1), \qquad k\to\infty.\]
On the other hand since $X_k$ grows exponentially by \eqref{eq:buglau}, there is $c_2=c_2(m,n)>1$ such that
\[X_k \gg c_2^{k}\]
or
\[\frac{ \log X_k}{ k } \ge \log c_2-o(1), \qquad k\to\infty.\]
Combining yields
\[\omega(\bxi)= \limsup_{k\to\infty} \frac{ -\log L_k }{\log X_k}=\limsup_{k\to\infty} \frac{ -\frac{\log L_k}{k} }{ \frac{\log X_k}{k} } \le \frac{- \log c_1}{ \log c_2}< \infty.\]
\qed 
\end{itemize}

 \subsubsection{First statement of Theorem \ref{optimn}}

Recall Lemma \ref{lem} shows that condition  (B) or (C)  implies that $h_{\bxi,1}$ and $h_{\bxi,2}$ remain at bounded distance. By connectedness (we assume the condition of \cite[Theorem 9.2]{Wgene}), there exists arbitrarily large $q$ such that $h_{\bxi,2}(q)=h_{\bxi,3}(q)= h_{\bxi,1}(q) +O(1)$. It follows from Minkowski that for such a $q$, one has
$$3 h_{\bxi,1}(q) \ge -(m+n-3) h_{\bxi,m+n}(q) +O(1).$$
Taking the limits, it gives
$$ 3 \, \limsup_{q\to\infty} \frac{h_{\bxi,1}(q)}{q} \ge -(m+n-3) \,\liminf_{q\to\infty} \frac{h_{\bxi,m+n}(q)}{q}.$$
Since $h_{\bxi,m+n}$ has slope at most $m$, according to \eqref{LimPOmega} we deduce
$$\frac{m - n\hat{\omega}(\bxi)}{1+\hat{\omega}(\bxi)} \ge \frac{-(m+n-3)m}{3}.$$
This is better than the trivial lower bound $\frac{m - n\hat{\omega}(\bxi)}{1+\hat{\omega}(\bxi)}\ge-n$ only for $m=1$ or $m=2$. For $m=1$ this proves $\hat{\omega}(\bxi) \le \frac{1}{2}$ and the optimality of the interval $[1/n,1/2]$ as in \cite{L}. For $m=2$, short computations provide the bound $\hat{\omega}(\bxi) \le 2$. \qed 
 
\subsubsection{Geometry of best approximations and second part of Theorem \ref{optimn}}
The proof is based on the study of the geometry of best approximations. On this topic, see also \cite{MaMo22}. Here we use essentially a construction from \cite{Szeged}. The key ingredient is the {\it monotonicity} lemma (see below Lemma \ref{monotonicity}). In \cite{Szeged} it was proven not for ordinary best approximation vectors but for so-called {\it best spherical approximations}. In the present paper we adapt the proof for ordinary approximations. An extended formulation of this lemma and its technical proof are shifted to the Appendix.\\

First of all, we should note that if all the best approximation vectors $ \bx_i$ eventually lie in a fixed three-dimensional subspace, then $\hat{\omega}(\bxi) \le 1$ and there is nothing to prove, as in this case the situation is reduced to approximations to a one-dimensional line.







 { Indeed, in the case under consideration $\pmb{\xi}$ is a $2\times n$ matrix and we deal with integer points close to the two-dimensional subspace  $\frak{L}$ defined in (\ref{frakel}).

If all the best approximations $\bx_i$ eventually belong to a three-dimensional rational subspace $\mathcal{B}$, then either
$ \frak{L} \subset \mathcal{B}$  (case 1$^0$)  or $\frak{L} \cap \mathcal{B}= \ell$ is a one-dimensional linear subspace (case 2$^0$).
But case 1$^0$ is not possible because of extra irrationality condition
(\ref{conti})
 in the second part of Theorem \ref{optimn}.
So let us consider  case 2$^0$. Vectors  $\bx_i$ are eventually  best approximation vectors from lattice $\Gamma = \mathcal{B}\cap \mathbb{Z}^d$
to the one-dimensional subspace
$\ell \subset \mathcal{B}$ with respect to the induced norm (see Section 3.3  from \cite{Mo1}). This means the following: as  $\bx_{i-1}$ and  $\bx_i$  are two successive best approximation vectors to $\pmb{\xi}$, the parallelepiped
$$
\Pi_i = \{ \bx= (\pmb{x},\pmb{y}), \pmb{x}\in \mathbb{R}^2,\,\,\pmb{y}\in \mathbb{R}^n:\,\, ||\pmb{x}||\le X_i,\,\, L_{\pmb{\xi}} (\bx) \le L_{i-1}\}
\subset \mathbb{R}^{2+n}
$$
contains only the integer points $\pmb{0}, \pm \bx_{i-1},  \pm \bx_{i}$
(and $\pm \bx_{i-1},  \pm \bx_{i}$ belong to its boundary),
and the convex, 0-symmetric three-dimensional set
$$
\Omega_i = \Pi_i \cap \mathcal{B}
$$
does not contain non-trivial  points of the lattice $\Gamma$ in its interior. 
As the intersection $\frak{L}\cap\mathcal{B}$ is a one dimensional subspace $\ell$,
there exist positive constants $c^\pm = c^\pm (\frak{L},\mathcal{B})$ such that for three-dimensional  sets
$$
\Omega_i^\pm = \{
\bx\in \mathcal{B}:\,\, ||\bx ||\le X_{i-1}, \,\,{\rm  dist}\, (\bx, \ell) \le c^\pm L_i
\}
$$
one has
$$
\Omega_i^-\subset \Omega_i \subset \Omega_i^+.
$$
In particular,
we have the bounds
$
{\rm vol}_3 \, \Omega_i  \asymp_{\frak{L}, \mathcal{B}} L_{i-1}^2X_i
$
for the three-dimensional volume of $\Omega_i$.
 Now consideration of two-dimensional lattice
 $$
 \Lambda_i = \langle  \bx_{i-1},\bx_i\rangle_\mathbb{Z}
 \subset \langle  \bx_{i-1},\bx_i\rangle_\mathbb{R}
 $$
 and two-dimensional section
 $$
 \Omega_i \cap\langle  \bx_{i-1},\bx_i\rangle_\mathbb{R}
 $$
 as in the  beginning of Subsection \ref{OptiOrdi1} 
 leads to (\ref{coovol}), and so we get 
  $ \hat{\omega}(\bxi) \le\omega(\bxi) \le 1$.

}







\vskip+0.3cm
So, we may assume that all the best approximations do not eventually belong to a three-dimensional subspace.

Consider an index  $i$ such that three consecutive best approximations $\bx_{i-1},\bx_i, \bx_{i+1}$ are linearly independent. We may assume that $\hat{\omega}(\bxi)>1$, to avoid the case when vectors $\frak{x}_1, \frak{x}_2$ considered in Subsection \ref{OptiOrdi1} in \eqref{pro1} are proportional, see Lemma 4 from \cite{Szeged} for the details). Considering determinants gives 
\[ L_{i-1} X_i X_{i+1} \gg 1 \]
which under (B) and (C) turns into
\begin{equation}\label{G1}
L_i \gg X_i^{-2}.
\end{equation}

Consider $\nu$ and $k$ indexes associated to two successive triples  
\begin{equation}\label{triplet}
\bx_{\nu-1},\bx_\nu, \bx_{\nu+1}\,\,\,\,\text{ and }\,\,\,\,\bx_{k-1},\bx_k, \bx_{k+1}
\end{equation}
 of linearly independent consecutive best approximations. As the best approximation vectors $\bx_i$ do not lie eventually in a three-dimensional subspace, {
there exist infinitely many successive triples  (\ref{triplet}) such that four vectors
 $\bx_{\nu-1},\bx_\nu, \bx_{\nu+1}, \bx_{k+1}$ are independent.
 In this situation all the vectors 
 \begin{equation}\label{triplet1}
\bx_\nu,\bx_{\nu+1},....,\bx_{k-1}, \bx_k
\end{equation}
belong to the same two-dimensional linear subspace and the four vectors
 \begin{equation}\label{triplet2}
\bx_{\nu-1},\bx_{k-1},\bx_{k-1}, \bx_k
\end{equation}
are independent. In this situation, in \cite{MaMo22} the successive triples \eqref{triplet} are said to form \emph{pattern B}, and in the terminology from \cite{Roy23}  the corresponding indices are successive elements of the set $J$}.\\

We consider the following \emph{monotonicity lemma}. 
 \begin{lemma}\label{monotonicity}
Suppose {that all the best approximation vectors \eqref{triplet1}} belong to the same two-dimensional linear subspace. 
Then
$$ L_\nu X_{\nu+1}\ll L_{k-1}X_k.$$
 \end{lemma}
 
 Let $1< \alpha < \hat{\omega}(\bxi)$. Then by \eqref{defexp}, for sufficiently large indexes $j$ we have $L_j < X_{j+1}^{-\alpha}$. By the monotonicity Lemma \ref{monotonicity}, and (B) we get
\begin{equation}\label{G2}
L_\nu \ll X_\nu^{-1} X_k^{1-\alpha}.
\end{equation}
Combining \eqref{G1} and \eqref{G2}, we deduce
\begin{equation}\label{G3}
X_\nu^{-2} \ll X_\nu^{-1} X_k^{1-\alpha}, \, \textrm{ that is } X_k^{ \alpha -1} \ll X_\nu.
\end{equation}

{
The four vectors (\ref{triplet2}) are linearly independent and so
dealing with 
coordinates
$$\bx_i =
(x_{1,i},x_{2,i}, {y}_{1,i},...,{y}_{n,i})^\top
$$
we 
follow the argument from \cite{Szeged}.
As $\alpha>1$ following Lemma 4 and the  beginning of the proof of Lemma 5 from \cite{Szeged}, 
we
see that  $4\times (2+n)$ matrix
$$
\left(
\begin{array}{ccccc}
x_{1,\nu-1}&x_{2,\nu-1}& {y}_{1,\nu-1}&...&{y}_{n,\nu-1}\cr
x_{1,k-1}&x_{2,k-1}& {y}_{1,k-1}&...&{y}_{n,k-1}\cr
x_{1,k}&x_{2,k}& {y}_{1,k}&...&{y}_{n,k}\cr
x_{1,k+1}&x_{2,k+1}& {y}_{1,k+1}&...&{y}_{n,k+1}
\end{array}
\right)
$$
has rank 4.
Considering determinants
(the argument here is completely similar to the proof of Lemma  5  from \cite{Szeged},
where all the calculations are very detailed), this provides the bound
\begin{equation}\label{newbou}
L_{\nu-1} L_{k-1}X_k X_{k+1} \gg 1.
\end{equation}
 }


  { Now we  combine (\ref{newbou}) with \eqref{G2}
  and  inequality $L_k < X_{k+1}^{-\alpha}$}
 to obtain
\begin{equation}\label{G4}
1 \ll X_\nu^{-1}  X_k^{3-2 \alpha} \, \textrm{ that is } X_\nu\ll X_k^{ 3 -2\alpha}.
\end{equation}
Hence comparing upper and lower bounds given by \eqref{G3} and \eqref{G4}, we get
\[X_k^{ \alpha -1} \ll X_k^{ 3 -2 \alpha}.\]
It follows that $\alpha \le \frac{4}{3}$, implying $\hat{\omega}(\bxi) \le \frac{4}{3}$. \qed

\section{Constructions claimed in Theorem$^\star$~\ref{Thm-}, Theorem$^\star$~\ref{ThmOrd}  and Theorem$^\star$~\ref{Main}}\label{Construction} \label{qq0}

In this section, all the constructions rely on the variational principle in parametric geometry of numbers (Theorem \ref{VarPrinc}). Note that if $\min(n,m)=1$, 
all the existence results may be deduced from Roy's theorem \cite{Roy}, without statement on Hausdorff or packing dimension.\\

The strategy of the constructions are always the same. We first construct an elementary $(m,n)$-template $\bP_k$ on a generic interval $I_k = [q_k, q_{k+1}]$, where $q_k>1$ is arbitrary and $q_{k+1}$ may depend on $q_k$. We then fix $q_1$ and concatenate these $(m,n)$-templates $\bP_k$ to obtain $\bP$ defined on the interval $\bigcup_{k\in\bN} I_k = [q_1, \infty]$. Then, $\bP$ is a $(m,n)$-template provided that the conditions (i) - (iv) are satisfied at each $q_k$. This constructs a family of templates $\cP$, where parameters may change on each concatenated elementary template.\\

It will be clear that the constructed templates $\bP$ are connected (see \cite[Theorem 9.2]{Wgene}, and this ensures total irrationality), and have strictly positive upper contraction rate ensuring existence of matrices via Theorem \ref{VarPrinc}. Moreover, the component $P_1$ of the templates $\bP$ satisfies \textbf{Condition U} from Lemma \ref{KeyLemma}. Then Proposition \ref{prop} provides the corresponding set of totally irrational matrices with associated properties. Note that  we even get the stronger property of having algebraically independent coordinates if the upper contraction rate (hence the packing dimension)  is strictly larger than $mn-1$, as explained in Remark \ref{RmkAlg}.

\subsection{Construction for Theorem$^\star$~\ref{Thm-}}\label{beau}

Fix dimensions $m,n\ge 1$, and $q_1$ large enough so that $q_1^2 \ge q_1 + \frac{m+n}{n} \log q_1$.\\
At $q_k$, we choose $P_1(q_k) = \cdots = P_{m+n}(q_k) = 0$. On the interval $[q_k, q_k + \log q_k]$, the components $P_1, \ldots , P_{m+n-1}$ are equal with joint slope $\sigma_c := \frac{-n m + (n-1)m}{m+n-1} = \frac{-m}{m+n-1} < 0$ while $P_{m+n}$ has slope $m$. Thus, 
\begin{eqnarray} \notag
P_1(q_k + \log q_k) &=& \cdots = P_{m+n-1}(q_k + \log q_k) = \frac{-n \log q_k}{m+n-1} \, \textrm{ and}\\ \label{t1}
P_{m+n}(q_k+\log q_k) &=& m \log q_k
\end{eqnarray}
Fix $t_k = q_k + \frac{m+n}{n} \log q_k$, on the interval $[q_k+\log q_k, t_k]$, the components $P_1, \ldots , P_{m+n-1}$ are equal with joint slope $\sigma_b := \frac{-n (m-1) + mn}{m+n-1} = \frac{n}{m+n-1} > 0$ while $P_{m+n}$ has slope $-n$. It is easy to check that 
\[ P_1(t_k) = \cdots = P_{m+n}(t_k) =0.\]
Fix $q_{k+1} = q_k^2 > t_k$, on the interval $[t_k,q_{k+1}]$ we simply choose our template to be trivial, that is $\bP = \boldsymbol{0}$. This constructs the elementary $(m,n)$-templates $\bP_k$ on $[q_k, q_{k+1}]$. The concatenation of these $(m,n)$-templates is a $(m,n)$-templates $\bP$, as conditions (i) - (iv) are satisfied at each $q_k$.

 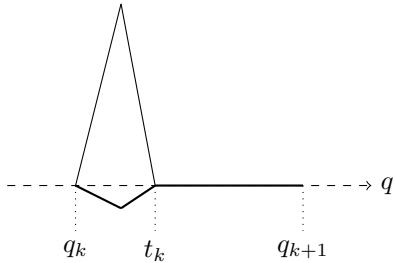
\begin{figure}[h!]
 \begin{center}
 \begin{tikzpicture}[scale=0.3]
 
    \draw[black, thin, dashed, ->] (-3,0) --(13,0) node [right, black] {$q$} ;
 
   \draw[black, thin] (2,8) --(0,0) node [below,black] {}   ;
   \draw[black, thick] (0,0) --(2,-1) ;
   \draw[black, thick] (2,-1) --(3.5,0) node [below,black] {$$} ;
   \draw[black, thin] (2,8) --(3.5,0) ; 
   \draw[black, thick] (3.5,0) --(10,0) node [below,black] {$$} ;
   
   \draw[black, dotted] (0,0) -- (0,-2) node [below, black] {$q_k$};
   \draw[black, dotted] (3.5,0) -- (3.5,-2) node [below, black] {$t_k$};
   \draw[black, dotted] (10,0) -- (10,-2) node [below, black] {$q_{k+1}$};

 \end{tikzpicture}
 \end{center}
 \caption{Example when $m=3$ and $n=4$. Bold lines represent more than 1 component.}\label{fig2}
 \end{figure}

It follows from \eqref{t1} that $\bP$ does not have (a). One can check the equivalences
\begin{eqnarray*}
 m> \frac{n}{m+n-1} =\sigma_b &\Leftrightarrow& m >1,\\
  -n < \frac{-m}{m+n-1}=\sigma_c &\Leftrightarrow& n >1.
 \end{eqnarray*}
Hence $\bP$ has properties (b) or (c) if respectively $m>1$ or $n>1$. Indeed, then $P_1$ never has extremal slopes $m$ and $-n$. 

One can easily adapt the construction to satisfy property (B) but not (C) or property (C) but not (B). Namely one can adapt the four main slopes of the construction to ensure either $\sigma_b=m, \sigma_c>-n$ or $\sigma_b<m,\sigma_c=-n$.

\begin{lemma}\label{LemContr}
$\underline{\delta}(\bP) = \overline{\delta}(\bP) = mn$.
\end{lemma}
By Proposition \ref{prop}, all statements of Theorem$^\star$~\ref{Thm-} follow.\\ 

Note that for this construction, as required to obtain full Hausdorff dimension, we have inferred exponents $\omega = \hat{\omega} = \frac{m}{n}$. 
\begin{proof}[Proof of Lemma \ref{LemContr}]
Since $q_k= o(q_{k+1})$ and $t_k = q_k + o(q_k)$, using only \eqref{Delta0} we can estimate the contraction rates
\begin{eqnarray*}
 \delta([q_1,q_{k+1}]) &=& \frac{ (q_k-q_1) \delta([q_1, q_k]) + (t_k - q_k) \delta([q_k,t_k]) + (q_{k+1}-t_k) mn  }{q_{k+1}-q_1}   \\
  &\to_{k\to\infty}& mn
\end{eqnarray*}
and for $q_k \le q \le t_k$
\begin{eqnarray*}
 \delta([q_1,q]) &=& \frac{ (q_k-q_1) \delta([q_1, q_k]) + (q - q_k) \delta([q_k,q]) }{q-q_1}   \\
  &\sim_{k\to\infty}& \delta([q_1, q_k]).
\end{eqnarray*}
That means, $\underline{\delta}(\bP) = \overline{\delta}(\bP) = mn$.
\end{proof}

\subsection{Constructions for ordinary exponents (Theorem$^\star$~\ref{ThmOrd})}\label{qq1}

To prove Theorem$^\star$~\ref{ThmOrd} we construct a family of templates $\mathcal{P}_{v,w}$ with two main integer parameters $v,w$. In Subsection \ref{ssa} we define all necessary parameters. In Subsection \ref{ssb} we construct the proper family  $\mathcal{P}_{v,w}$ of templates.
Choosing parameters accordingly will allow to consider all cases of Theorem$^\star$~\ref{ThmOrd}: we deal with property (B) in \ref{ParamOrdB}, property (C) in \ref{ParamOrdC} and both property (B) and (C) in \ref{ParamOrdBC}.\\

\subsubsection{Setting parameters}\label{ssa}
Dimensions $m\ge1$ and $n\ge1$ are fixed. Choose two main integer parameters { $0\le v < n$} and $1 \le w \le \min(m,n-v)$. Consider the main slopes
\begin{eqnarray}
\sigma_c = \sigma_c (v)&:= &\frac{-nm + mv}{m+v},\\
\sigma_b = \sigma_b (v,w) &:= &\frac{-n(m-w) +m(v+w)}{m+v}.
 \end{eqnarray}
 Note that $\sigma_c$ is always strictly negative while $\sigma_b$ may have any sign. Given $\gamma = n-v-2w$, consider the auxiliary slopes
\begin{eqnarray}
\sigma_1 &:=& \begin{cases} -n &\textrm{ if } \gamma \le 0 \\ \frac{-wn +\gamma m}{n-w-v} &\textrm{ if } \gamma > 0 \end{cases},\\
\sigma_2 &:=& \begin{cases}\frac{\gamma n + (w+\gamma) m}{w} & \textrm{ if } \gamma < 0\\  m &\textrm{ if } \gamma \ge 0 \\  \end{cases},\\
\sigma_3 &:= & \frac{-wn +(n-v-w)m}{n-v}.
\end{eqnarray}

For an index $k \in \mathbb{N}$, consider a pair of parameters $\lambda_k \in [\sigma_c,\min(0,\sigma_b) ]$ and $\mu_k := \frac{m+v+w}{n-v-w} \lambda_k \ge \lambda_k$. Choose $q_k \gg 1$ and $p_k\ge q_k +\log(q_k)$. {This choice determines} $$p_k' := p_k\left(1+ \frac{\mu_k-\lambda_k}{\sigma_1-\sigma_2}\right)\ge p_k$$ and $$q_k^{\max}=\begin{cases} \infty \textrm{ if } \sigma_b \le 0 \\ p'_k + \frac{(m-\sigma_c)(p_k-q_k) + (m-\sigma_2)(p'_k-p_k)}{\sigma_3-\sigma_b} \end{cases}.$$ Choose $r_k\in[p'_k , q_k^{\max}]$. Finally, $q'_k$ is determined by $$q'_k:= r_k + \frac{(p_k-q_k)(m+n) + (p_k'-p_k)(\sigma_2-\sigma_b) + (r_k-p_k')(\sigma_3-\sigma_b)}{n+\sigma_b}.$$

Parameters satisfying the conditions of this section are called \emph{admissible}.\\

 {Here slopes are denoted by $\sigma$ and parameters $\lambda_k$ and $\mu_k$ describe the configuration at $q_k$. The three parameters $p_k$, $q_k$ and $r_k$ are freely chosen in some intervals while $p_k'$, $q_k'$ and $q_k^{\max}$ are uniquely determined by the choice of the previous triple.}

 \subsubsection{Construction}\label{ssb}
Given admissible main integer parameters $v,w$ and real parameters $\lambda_k,q_k,p_k,r_k$, we first construct on the segment $[q_k,q_k']$ an elementary template $\bP_{v,w}(\lambda_k,q_k,p_k,r_k)$. We explain after the construction how they can be concatenated to form a family of templates $\cP_{v,w}$ depending on both main parameter $v,w$ and a pair of sequences $(p_k)_{k\in \bN}, (r_k)_{k\in \bN}$.


   Given the initial condition at $q_k$, it is enough to describe the behavior of the slopes of each components of $\bP_{v,w}(\lambda_k,q_k,p_k,r_k) = (P_1,..., P_{m+n})$ on all subintervals. The constructed template has a different shape depending on the sign of $\sigma_b$ and is depicted in Figures \ref{FigOrd1} ($\sigma_b>0$) and Figure \ref{FigOrd2} ($\sigma_b\le0$).\\

At $q_k$, the initial condition is set to be $P_1(q_k)= \cdots = P_{m+v+w}(q_k) = \lambda_k q_k$ and $P_{m+v+w+1}(q_k) = \cdots = P_{m+n}(q_k) = \mu_k q_k$. By definition of $\mu_k$, the sum $\sum_{i=1}^{m+n} P_i(q_k) =0$. We describe slopes on all subintervals:
\begin{itemize}
\item On $[q_k,p_k]$, $P_1= \cdots = P_{m+v}$ have slope $\sigma_c$ and both $P_{m+v+1}= \cdots = P_{m+v+w}$ and $P_{m+v+w+1} = \cdots = P_{m+n}$ have slope $m$. 
\item  On $[p_k,q'_k]$, $P_1 = \cdots = P_{m+v}$ have slope $\sigma_b$, while the other components have a slightly more subtle behavior, namely : 
\begin{itemize}
\item On $[p_k,p_k']$, $P_{m+v+w+1} = \cdots = P_{m+n}$ have slope $\sigma_1$ and $P_{m+v+1}=\cdots = P_{m+v+w}$ have slope $\sigma_2$. Note that $p_k'$ is defined so that $P_{m+v+1}(p_k') = P_{m+n}(p_k')$.
\item On $[p_k',r_k]$, $P_{m+v+1} = \cdots = P_{m+n}$ have slope $\sigma_3$.
\item On $[r_k,q_k']$, $P_{m+v+w+1} = \cdots = P_{m+n}$ have slope $m$ and $P_{m+v+1}=\cdots = P_{m+v+w}$ have slope $-n$.
\end{itemize}
\end{itemize}

Note that $q_k'$ was defined purposely to have $P_1(q_k')= \cdots = P_{m+v+w}(q_k') : = \lambda_k' q_k'$, where $\min(0,\sigma_b) \ge \lambda_k' \ge \sigma_c$.
  So we constructed  $\bP_{v,w}(\lambda_k,q_k,p_k,r_k)$ on the interval $[q_k,q_k']$.\\
  
  We now concatenate this elementary templates. Note that $q_k'$ and $\lambda_k'$ are determined by the choice of $p_k$ and $r_k$. Hence, fixing $q_1 \gg1$ and $\lambda_1 \in  [\sigma_c,\min(0,\sigma_b) ]$, we can construct inductively a template $\bP_{v,w}((p_k)_{k\in \bN}, (r_k)_{k\in \bN})$ given by two sequences $(p_k)_{k\in \bN}, (r_k)_{k\in \bN}$ satisfying at each step $k$ the conditions of previous section and $p_{k+1}\ge q_{k+1}:=q_k'$. The freedom of this two sequences determines for all $k\in\bN$ the values of $\lambda_{k+1}$ and $q_{k+1}$. We denote by $\cP_{v,w}$ the set of all templates $\bP_{v,w}((p_k)_{k\in \bN}, (r_k)_{k\in \bN})$ for admissible sequences $(p_k)_{k\in \bN}, (r_k)_{k\in \bN}$.

\begin{figure}[h!]
 \begin{center}
 \begin{tikzpicture}[scale=0.5]
 
 \draw[black, thick] (0,0) -- (2,-4) node [midway,below] {$\sigma_c$};
 \draw[black,thick] (0,0) -- (2,4);
  \draw[black,thick] (0,3) -- (2,7);
 
 \draw[black, thick] (2,4) -- (4,5) node [midway,above] {$\sigma_2$};
 \draw[black, thick] (2,7) -- (4,5) node [midway,above] {$\sigma_1$};
 
 \draw[black,thick] (4,5) -- (6,4) node [midway,above] {$\sigma_3$};
 
  \draw[black,thick] (2,-4) -- (8,-1) node [midway,above] {$\sigma_b$};
  \draw[black,thick] (6,4) -- (8,-1) ;
  \draw[black,thick] (6,4) -- (8,8) ;

 \draw[black, thin, ->] (-1,1) -- (14,1);
 
  \draw[black, dotted, thin] (6,4) -- (13,1);
  \draw[black, dotted, thin] (8,-1) -- (13,1);
 
  \draw[black, dashed, thin] (0,3) -- (0,-5) node [below] {$q_k$};
   \draw[black, dashed, thin] (2,7) -- (2,-5) node [below] {$p_k$};
   \draw[black, dashed, thin] (4,5) -- (4,-5) node [below] {$p'_k$};
   \draw[black, dashed, thin] (6,4) -- (6,-5) node [below] {$r_k$};
   \draw[black, dashed, thin] (8,8) -- (8,-5) node [below] {$q'_k$};
   \draw[black, dashed, thin] (13,1) -- (13,-5) node [below] {$q_k^{\max}$};

 \end{tikzpicture}
 \end{center}
 \caption{Construction of $\bP_{v,w}(p_k,r_k)$ when $\sigma_b > 0$. Slopes different from $-n$ and $m$ are labelled.}\label{FigOrd1}
 \end{figure}

\begin{figure}[h!]
 \begin{center}
 \begin{tikzpicture}[scale=0.5]
 
 \draw[black, thick] (0,0) -- (2,-4) node [midway,below] {$\sigma_c$};
 \draw[black,thick] (0,0) -- (2,4);
 \draw[black,thick] (0,3) -- (2,7);
 \draw[black, thick] (2,7) -- (3,5) node [midway,above] {$\sigma_1$};
 \draw[black, thick] (2,4) -- (3,5) node [midway,above] {$\sigma_2$};
 \draw[black,thick] (3,5) -- (6,6) node [midway,above] {$\sigma_3$};
 \draw[black,thick] (2,-4) -- (9,-5) node [midway,above] {$\sigma_b$};
 \draw[black,thick] (6,6) -- (9,-5);
 \draw[black,thick] (6,6) -- (9,12) ;
 
  \draw[black, thin, ->] (-1,1) -- (14,1);
 
   \draw[black, dashed, thin] (0,3) -- (0,-6) node [below] {$q_k$};
   \draw[black, dashed, thin] (2,7) -- (2,-6) node [below] {$p_k$};
   \draw[black, dashed, thin] (3,5) -- (3,-6) node [below] {$p'_k$};
   \draw[black, dashed, thin] (6,6) -- (6,-6) node [below] {$r_k$};
   \draw[black, dashed, thin] (9,12) -- (9,-6) node [below] {$q'_k$};
   \
 
 \end{tikzpicture}
 \end{center}
 \caption{Construction of $\bP_{v,w}(p_k,r_k)$ when $\sigma_b\le0$. Slopes different from $-n$ and $m$ are labelled.}\label{FigOrd2}
 \end{figure}

\subsubsection{Properties of templates from the family $\mathcal{P}_{v,w}$ and proof of Theorem$^\star$~\ref{ThmOrd}}\label{qq2}
The next lemma provides useful properties of elements of the family of templates $\mathcal{P}_{v,w}$.

\begin{lemma}\label{ConstrLemma}
Fix admissible main parameters $v,w$. When $\bP$ ranges through $\mathcal{P}_{v,w}$, then $\liminf_{q\to \infty} \frac{P_1(q)}{q}$ ranges through the full interval $[\sigma_c, \min(0,\sigma_b)]$.
\end{lemma}

\begin{proof}
{ Note that $P_1$ has slope either $\sigma_c$ or $\sigma_b$, hence $\frac{P_1(q)}{q} \in [\sigma_c, \min(0,\sigma_b)]$. One can always choose $r_k$ large enough so that $\lambda_{k+1} = \min(0,\sigma_b) +o(1)$. For $p_k=q_k +\log(q_k)$, we have $\frac{P_1(p_k)}{p_k} = \lambda_k +o(1) = \min(0,\sigma_b) +o(1)$, while for $p_k \gg q_k$, we have $\frac{P_1(p_k)}{p_k} = \sigma_c +o(1)$. Since our parameters are allowed to vary continuously, $\liminf_{q\to \infty} \frac{P_1(q)}{q}$ ranges through the full interval $[\sigma_c, \min(0,\sigma_b)]$. }
\end{proof}

The next lemma explains the necessary and sufficient conditions on the parameters of the constructed elements of $\cP_{v,w}$ to satisfy conditions (b) and/or (c).

\begin{lemma}\label{ConstrLemma2}
Let  $\bP \in \mathcal{P}_{v,w}$. Then  
\begin{itemize}
\item $\bP$ satisfies (b) $\Leftrightarrow \sigma_b <m \Leftrightarrow w < m $,
\item $\bP$ satisfies (c) $\Leftrightarrow \sigma_c >-n \Leftrightarrow v\ge1$,
\item $\bP$ does not satisfy (a).
\end{itemize}
In addition 
\begin{itemize}
\item If $\sigma_b>0$, there exists another family of templates 
$\tilde{\mathcal{P}}_{v,w}$ such that 
$$
\{ \bxi\in \cM(\bP ): \bP \in \tilde{\mathcal{P}}_{v,w}\}
$$
 have full packing dimension.
\end{itemize}
\end{lemma}

\begin{proof}
For the last observation, note that under this condition we can choose $\lambda_k = 0$ and $r_k = q_k^{\max}=q_k'$ for every elementary template. Hence we can extend these elementary templates by the trivial template on intervals $[q_k^{\max},(q_k^{\max})^2]$ and set $q_{k+1}=(q_k^{\max})^2$. This constructs a variant family of templates $\tilde{\mathcal{P}}_{v,w}$. Then observations similar to the proof of Lemma \ref{LemContr} provide full packing dimension.

Property (a) is prevented by $p_k \ge q_k + \log(q_k)$. The rest of the proof follows from the construction because the first component of the template has slopes either $\sigma_b$ or $\sigma_c$, on intervals of arbitrarily large length.
\end{proof}

\begin{remark} Easy computations show that if there exists $\alpha,\beta$ such that $\liminf_{t\to\infty} \frac{P_1(t)}{t} = \frac{-\alpha n +\beta m}{\alpha + \beta}$, then by \eqref{LimPOmega} the associated ordinary exponent is $\frac{\alpha}{\beta}$. Also, in the constructions considered, choosing $r_k$ to ensure $\lambda_k = 0$ provides non-singular matrices $\bxi\in \cM(\bP)$. Otherwise, we may always choose $r_k$ to ensure $\lambda_k \ge \log(q_k)$ and obtain singular matrices $\bxi\in \cM(\bP)$ (obviously loosing full packing dimension).
\end{remark}

\subsubsection{First statement of Theorem$^\star$~\ref{ThmOrd}}\label{ParamOrdB}
Recall that $m\ge2$. We want to construct elements in the set $S_{B,n,m}(\omega)$. By Lemma \ref{ConstrLemma2}, templates from the constructed family $\cP_{v,w}$ have property (b) when $w<m$. To obtain the full admissible interval $\omega \in [\frac{m}{n}, \infty]$, depending on the dimensions $m$ and $n$, we may have to consider various pairs of parameters $(v,w)$.\\

We first chose parameters $v=0$ and $w=\min(m-1,n)<m$. Observe that with this choice, 
\[\begin{cases}
\sigma_c &= -n\\[2mm]
{ m > \sigma_b }&{= \begin{cases} \frac{-n+ (m-1)m}{m} \textrm{ if } w=m-1 \\ \frac{n^2}{m} \textrm{ if } w=n\end{cases}}
\end{cases}.\]
In particular, $\sigma_b>0$ if $n<m(m-1)$. Hence if $n<m(m-1)$, by Lemmata \ref{ConstrLemma} \& \ref{ConstrLemma2} , then templates of the family $\cP_{0,\min(m-1,n)}$ have property (b) and $\liminf_{q\to \infty} \frac{P_1(q)}{q}$ reaches any value in $[-n, 0]$. By \eqref{LimPOmega}, the associated ordinary exponent is ranging through $[\frac{m}{n},\infty]$. 
Hence for any fixed exponent $\omega\in [\frac{m}{n},\infty]$, via Proposition \ref{prop} we obtain a set of matrices of full packing dimension, with ordinary exponent $\omega$ and satisfying (B) but not (A).\\
Thus, under condition $n<m(m-1)$, the full interval $[\frac{m}{n}, \infty]$ is reached with a single pair of parameters $(v,w) = (0,\min(m-1,n))$. Otherwise, we need more.\\

Suppose now that $n\ge m(m-1)$. First, consider the same pair of parameters $(v,w) = (0,\min(m-1,n))$. By Lemmata \ref{ConstrLemma} \& \ref{ConstrLemma2}, templates in the family $\cP_{0,\min(m-1,n)}$ have 
\begin{equation}\label{part3}
\liminf_{q\to \infty} \frac{P_1(q)}{q} \textrm{ ranges through } [-n,  \frac{-n+ (m-1)m}{m}]
\end{equation}
 and have property (b). Via Proposition \ref{prop} and \eqref{LimPOmega}, for any fixed exponent $\omega \in [\frac{1}{m-1},\infty]$, we obtain a set of matrices with ordinary exponent $\omega$ and satisfying (B) but not (A). Note that if the exact equality $n = m(m-1)$, we have $\frac{m}{n} = \frac{1}{m-1}$ and we get the desired interval - but without full packing dimension since $\sigma_b=0$. \\

Now we consider more pairs of parameters $(v,w)$ to fill the gap $[\frac{m}{n},\frac{1}{m-1}]$. Consider $l_n\ge 2$ such that $m^{l_n-1}-m < n \le m^{l_n} -m$. Consider the family of pairs of parameters indexed by $l\in[2, l_n]$ given by
\[\begin{cases}
v_l &= m^l-m\\
w_l &= \min(m-1,m+n-m^l) < m
\end{cases}.\]

Note that by definition of $l_n$, for $l<l_n$ we always have $w_l = \min(m-1,m+n-m^l) = m-1$. In this case, the family of templates $\cP_{m^l-m,m-1}$ with
\[\begin{cases}
\sigma_{c,l} &= \frac{-mn +m(m^l-m)}{m+m^l-m} = \frac{-n+m^l-m}{m^{l-1}} > -n \\[2mm]
\sigma_{b,l} &= \frac{-n+m(m^l-m+m-1)}{m+m^l-m} = \frac{-n+m^{l+1}-m}{m^l} < m
\end{cases}\]
provides by Lemmata \ref{ConstrLemma} \& \ref{ConstrLemma2} templates satisfying (b) but not (a) with 
\begin{equation}\label{part1}
\liminf_{q\to\infty} \frac{P_1(q)}{q} \textrm{ ranging through }\left[\frac{-n+m^l-m}{m^{l-1}}, \frac{-n+m^{l+1}-m}{m^l}\right].\end{equation}
Consider now the case $l=l_n$, and the family of templates $\cP_{v_{l_n},w_{l_n}}$, it has 
\[\begin{cases}
\sigma_{c,l_n} &= \frac{-n+m^l_n-m}{m^{l_n-1}} \\[4mm]
\sigma_{b,l_n} &= \begin{cases} \frac{n(n-(m^{l_n-m}-m))}{m^{l_n}} > 0  \textrm{ if } \min(m-1,m+n-m^{l_n})=n+m-m^{l_n} \\ \\
 \frac{-n+m^{l_n+1}-m}{m^{l_n}} \ge 0 \textrm{ if } \min(m-1,m+n-m^{l_n}) = m-1
\end{cases}
\end{cases}.\]  
and
\begin{equation}\label{part2}
\liminf_{q\to\infty} \frac{P_1(q)}{q} \textrm{ ranges through }\left[\frac{-n+m^{l_n}-m}{m^{l_n-1}} ,0\right].
\end{equation}
Hence, considering all these pairs of parameters, by Lemmata \ref{ConstrLemma} \& \ref{ConstrLemma2} we are able to construct templates satisfying (b) but not (a), with $\liminf_{q\to\infty} \frac{P_1(q)}{q}$ ranging through the union of intervals \eqref{part3},\eqref{part1},\eqref{part2} :
 $$[-n, 0] = \left[-n,  \frac{-n+ (m-1)m}{m}\right] \cup  \bigcup_{l=2}^{l_n-1} \left[\frac{-n+m^l-m}{m^{l-1}}, \frac{-n+m^{l+1}-m}{m^l}\right] \cup \left[\frac{-n+m^{l_n}-m}{m^{l_n-1}} ,0\right].$$
Via Proposition \ref{prop} and \eqref{LimPOmega}, for any fixed exponent $\omega\in [\frac{m}{n},\infty]$, we obtain a set of matrices $\bxi$ with ordinary exponent $\omega(\bxi)=\omega$ satisfying (B) but not (A).

 Note that our constructed templates provide strictly positive Hausdorff and packing dimension.\\
\begin{remark}\label{Rmk1}
For $2\le l \le l_n$, the templates of the family $\cP_{v_l,w_{l}}$ also satisfy (c) because $v_l \ge1$. This observation will be used for the proof of the Third statement of Theorem \ref{ThmOrd}.
\end{remark}

\subsubsection{Second statement of Theorem$^\star$~\ref{ThmOrd}}\label{ParamOrdC}

Recall that $n\ge2$. We want to construct elements of $S_{C,n,m}(\omega)$. By Lemma \ref{ConstrLemma2}  constructed templates of the family $\cP_{v,w}$ have property (c) under the condition $v \ge1$. Here a single pair of main parameters $v,w$ is enough.\\
Chose main parameters $v=1$ and $w=\min(m,n-1)$. Observe that with this choice, 
\[\begin{cases}
\sigma_c &= \frac{-n m + m}{m+1}> -n \\[2mm]
\sigma_b &= \begin{cases} {m >0 \textrm{ if } w=m }\\{ \frac{n(n-1)}{m+1} >0 \textrm { if } w=n-1} \end{cases}
\end{cases}.\]
Hence by Lemmata \ref{ConstrLemma} \& \ref{ConstrLemma2}, templates of the family $\cP_{1,\min(m,n-1)}$ have $\liminf \frac{P_1(q)}{q}$ reaching any value in $[\frac{-n m + m}{m+1}, 0)$ and has property (c) but not (a). Via Proposition \ref{prop} and \eqref{LimPOmega}, for any fixed exponent $\omega'\in [\frac{m}{n},m]$, we obtain a set of matrices of full packing dimension, with ordinary exponent $\omega'$ and satisfying (C) but not (A).

\subsubsection{Third statement of Theorem$^\star$~\ref{ThmOrd}}\label{ParamOrdBC}
 Recall that $\min(m,n)\ge 2$.  We want to construct elements in the intersection $S_{B,n,m}(\omega) \cap S_{C,n,m}(\omega)$. By Lemma \ref{ConstrLemma2} templates of the family $\cP_{v,w}$ have property (b) and (c) under the conditions $w<m$ and $v \ge1$. Again, various pairs of parameters $(v,w)$ are requested, depending on dimensions, to reach the full claimed interval.\\

Chose first main parameters $v=1$ and $w=\min(m-1,n-1)<m$. Observe that with this choice, 
\[\begin{cases}
\sigma_c &= \frac{-mn +m}{m+1} > -n \\[4mm]
\sigma_b &=  \begin{cases} \frac{-n +m^2}{m+1}  \textrm{ if } \min(m-1,n-1)=m-1 \\ \\
 \frac{-n(m-n+1) +mn}{m+1} = \frac{n^2-1}{m+1} > 0 \textrm{ if } \min(m-1,n-1) = n-1
\end{cases}
\end{cases}.\]
Further, note that $\sigma_b >0$ if and only if $n<m^2$. Hence if $n<m^2$, by Lemmata \ref{ConstrLemma} \& \ref{ConstrLemma2}, templates of the family $\cP_{1,\min(m-1,n-1)}$ have properties (b) and (c) and $\liminf_{q\to\infty} \frac{P_1(q)}{q}$ reaches any value in $[\frac{-mn +m}{m+1}, 0)$. That is, by \eqref{LimPOmega}, the associated ordinary exponent is ranging through $[\frac{m}{n},m]$. 
Hence for any fixed exponent $\omega'\in [\frac{m}{n},m]$, via Proposition \ref{prop} we obtain a set of matrices of full packing dimension, with ordinary exponent $\omega'$ and satisfying (B) and (C) but not (A).\\

If $n\ge m^2 > m^2 -m$, previous choice of main parameters $(v,w)=(1, m-1)$ provides a family of templates $\cP_{1,m-1}$ with 
\begin{equation}\label{Part1}
\liminf_{q\to\infty} \frac{P_1(q)}{q} \textrm{ ranging through }\left[ \frac{-mn+m}{m+1},\frac{-n+m^2}{m+1} \right].\end{equation}
 Another choice of main parameters $(v,w)=(2, m-1)$ provides a family of templates $\cP_{2,m-1}$ with 
 \begin{equation}\label{Part2}
 \liminf_{q\to\infty} \frac{P_1(q)}{q} \textrm{ ranging through } \left[ \frac{-mn+2m}{m+2},\frac{-n+m^2+m}{m+2} \right].\end{equation}
 Then, consider again $l_n\ge 2$ such that $m^{l_n-1}-m < n \le m^{l_n} - m$, we showed (see Remark \ref{Rmk1}) that templates with parameters $(v,w)=(v_l,w_l) =( m^{l}-m,\min(m-1,m+n-m^{l}))$ with $2\le l \le l_n$ satisfy both (b) and (c) but not (a), with 
 \begin{equation}\label{Part3}
 \liminf_{q\to\infty} \frac{P_1(q)}{q}\textrm{ ranging through }\left[\frac{-n+m^2-m}{m}, 0\right).\end{equation}
To summarize, we obtain families of templates satisfying both (b) and (c) but not (a) with $\liminf_{q\to\infty} \frac{P_1(q)}{q}$ ranging through the union \eqref{Part1},\eqref{Part2},\eqref{Part3}:
$$ \left[ \frac{-mn+m}{m+1},0\right) = \left[ \frac{-mn+m}{m+1},\frac{-n+m^2}{m+1} \right] \cup  \left[ \frac{-mn+2m}{m+2},\frac{-n+m^2+m}{m+2} \right] \cup \left[\frac{-n+m^2-m}{m}, 0\right) $$
By \eqref{LimPOmega}, the associated ordinary exponent is ranging through $[\frac{m}{n},m]$. Hence for any fixed exponent $\omega'\in [\frac{m}{n},m]$, via Proposition \ref{prop} we obtain a set of matrices with ordinary exponent $\omega'$ and satisfying (B) and (C) but not (A).\\
\qed

\subsection{Constructions for uniform exponents (Theorem$^\star$~\ref{Main})}\label{qq3}
To secure that $\hat{\omega}$ is as requested, in lights of \eqref{LimPOmega} we construct elementary {$(m,n)$-templates}  $\bQ_k$ on $[q_k, q_{k+1}]$ such that $Q_1(q_k) = \lambda_k q_k$ are the local maxima and  $\lambda_k \to_{k\to\infty} \frac{m-n\hat{\omega}}{1+\hat{\omega}}$. For some $q_k < p_k < q_{k+1}$, $Q_1$ has negative slope on $[q_k,p_k]$ and positive or constant slope on $[p_k,q_k]$.\\
Here, we consider different templates for construction related to (b) or (c).\\

\subsubsection{ Construction related to property (B) }\label{qq5}
Set $u = \min(\lfloor\frac{m+1}{2}\rfloor,n+1) \ge 2$ and the main parameter $0\le v \le n +1-u$. The requirement $u\ge2$ implies in particular $m\ge 3$. The case $m=2$ is treated as a byproduct in the next section.\\

Consider the main slopes 
\begin{eqnarray}
\sigma_c &= &\frac{-nu + mv}{u+v} ,\\
\sigma_b &= &\frac{-n +m(u+v-1)}{u+v} <m
 \end{eqnarray}
 and the auxiliary slopes
$$\begin{array}{ll}
\sigma_1 =  \frac{-n(m-u) + m (n-v-u+1)}{m+n -2u-v+1},   & \sigma_1' = \frac{-n(m-2u+1) + m (n-v-u+1)}{m+n -3u-v+2}   ,\\[4mm]
\sigma_2 = \frac{-n(m-u) +m(n-v)}{m+n-u-v} ,   & \sigma_2' =  \frac{-n(m-2u+1) +m(n-v)}{m+n-2u-v+1}  ,\\[4mm]
\sigma_3 = \frac{-n(m+1) + m(n+1-u-v)}{m+n - u -v }  , & \sigma_3' =  \sigma_1. \\
\end{array}$$

For an index $k \in \mathbb{N}$, consider $\lambda_k\in [\sigma_c, \min(0,\sigma_b) ]$ and $\mu_k := \frac{2u+v-1}{m+n-2u-v+1 } \lambda_k \ge \lambda_k$. { In the construction, $\lambda_k$ and $\mu_k$ describe the configuration at $q_k$.}  Choose freely $q_k \gg 1$, $p_k\ge q_k +\log(q_k)$ and $r_k\ge p_k$ not too large.\\

Given these parameters, we construct an elementary template $\bQ_v(q_k,p_k,r_k)$ on the interval $[q_k, r_k]$. Given the initial condition at $q_k$, it is enough to describe the behavior of the slopes on all subintervals. The constructed template is depicted in Figure \ref{FigUnif1}.\\

We first construct a base of our template as follows, before some adjustments.\\
At $q_k$, the initial condition is set to be $Q_1(q_k)= \cdots = Q_{2u+v-1}(q_k) = \lambda_k q_k$ and $Q_{2u+v}(q_k) = \cdots = Q_{m+n}(q_k) = \mu_k q_k$. By definition of $\mu_k$, the sum $\sum_{i=1}^{m+n} Q_i(q_k) =0$. We describe slopes on all subintervals:
\begin{itemize}
\item On $[q_k,p_k]$, $Q_1= \cdots = Q_{u+v}$ have slope $\sigma_c$ while $Q_{u+v+1}= \cdots = Q_{2u+v-1}$ have slope $m$ and $Q_{2u+v} = \cdots = Q_{m+n}$ have slope $\sigma_1 <m$ until they meet and $Q_{u+v+1}= \cdots = Q_{m+n}$ have slope $\sigma_2$. 
\item  On $[p_k, r_k]$, $Q_1= \cdots = Q_{u+v}$ have slope $\sigma_b$ while $Q_{u+v+1}= \cdots = Q_{m+n}$ have slope $\sigma_3$. 
\end{itemize}
Here, we assume implicitly that $r_k$ is small enough if $\sigma_b \ge \sigma_3$, so that $Q_1(r_k) \le Q_{m+n}(r_k)$ remains, it is not relevant to provide the explicit bound. We now modify our base template: consider the line of slope $-n$ passing through the point $(r_k, Q_1(r_k))$. It intersects $Q_{m+n}$ at some abscissa $r'_k$. On the interval $[r_k',r_k]$, $u-1$ components are represented by this line, while the other components $P_j$ with index $j \ge u+v$ have their slope $\sigma_i$ changed into $\sigma_i'$, for $1\le i \le 3$.\\

\begin{figure}[h!]
 \begin{center}
 \begin{tikzpicture}[scale=0.5]
 
 \draw[black, thick] (0,-3) -- (7,-8) node [midway,below] {$\sigma_c$};
  \draw[black,thick] (7,-8) -- (9,-7) node [midway,above] {$\sigma_b$};
  
  \draw[black,thick] (0,-3) -- (2,0);
   \draw[black,thick] (0,1) -- (2,0) node [midway,above] {$\sigma_1$};
    
   \draw[black,thick] (2,0) -- (5,1) node [midway,below] {$\sigma_2$};

   \draw[black,thick] (9,-7) -- (5,1);
   
   \draw[black,thick] (7,3) -- (5,1) node [midway,below] {$\sigma_2'$};
   
   \draw[black,thick] (7,3) -- (9,2) node [midway,below] {$\sigma_3'$};

   \draw[black, thin, dotted, ->] (-1,-2) -- (14,-2);
 
  \draw[black, dashed, thin] (0,3) -- (0,-9) node [below] {$q_k$};
   \draw[black, dashed, thin] (5,3) -- (5,-9) node [below] {$r'_k$};
   \draw[black, dashed, thin] (9,3) -- (9,-9) node [below] {$r_k$};
   \draw[black, dashed, thin] (7,3) -- (7,-9) node [below] {$p_k$};

 \end{tikzpicture}
 \end{center}
 \caption{Construction of $\bQ_{v}(q_k,p_k,r_k)$ when $\sigma_b\ge 0$. Slopes different from $-n$ and $m$ are labelled.}\label{FigUnif1}
 \end{figure}
 
 This finishes the construction of $\bQ_v(q_k,p_k,r_k)$ on the interval $[q_k, r_k]$.\\

Fix an admissible main parameter $v$. Setting $q_1\gg1$ and $\lambda_1\in [\sigma_c, \min(0,\sigma_b) ]$, the choice of a pair of admissible sequences $(p_k)_{k\in\bN},(r_k)_{k\in\bN}$ and fixing $q_{k+1}=r_k$ allows to construct inductively a template $\bQ_v((p_k)_{k\in\bN},(r_k)_{k\in\bN})$ by concatenation on the interval $[q_1,\infty)$. Note that $\lambda_{k+1}$ is determined by the choice of $r_k$. We denote by $\cQ_v$ the family of all templates $\bQ_v((p_k)_{k\in\bN},(r_k)_{k\in\bN})$ for admissible sequences $(p_k)_{k\in\bN},(r_k)_{k\in\bN}$. It satisfies an analogue to Lemma \ref{ConstrLemma} and Lemma \ref{ConstrLemma2}.\\

\begin{lemma}\label{ConstrLemma3}
Fix  an integer $v$ admissible, for a template $\bQ \in\cQ_{v}$ constructed above,
\begin{itemize}
\item  When $\bQ$ ranges through $\cQ_v$, then $\limsup_{q\to \infty} \frac{Q_1(q)}{q}$ ranges through the full interval $[\sigma_c, \min(0,\sigma_b)]$,
\item $\bQ$ satisfies (b)
\item $\bQ$ satisfies (c) $\Leftrightarrow \sigma_c > -n \Leftrightarrow v\ge1$,
\item $\bQ$ does not satisfy (a),
\end{itemize}
\end{lemma}

\begin{proof}
Note that $Q_1$ has slope either $\sigma_c$ or $\sigma_b$, hence $\frac{Q_1(q)}{q} \in [\sigma_c, \min(0,\sigma_b)]$. If we chose $p_k \gg q_k= r_{k-1}$ and $r_k = p_k +O(q_k)$, then  $\frac{Q_1(q_k)}{q_k} = \sigma_c + o(1)$. On the other hand, if we chose $p_k=O (q_k)$ and $r_k \gg p_k$, then  $\frac{P_1(q_k)}{q_k} = \min(0,\sigma_b + o(1))$. Since our parameters are allowed to vary continuously, $\limsup_{q\to \infty} \frac{Q_1(q)}{q}$ ranges through the whole $[\sigma_c, \min(0,\sigma_b)]$.
The other properties follow directly from the minoration of $p_k$ and the explicit values of $\sigma_b$ and $\sigma_c$.
\end{proof}

We now chose parameter $v$ to prove the first statement of Theorem \ref{Main}.\\
 
If $u=n+1 \le \lfloor\frac{m+1}{2}\rfloor$, choose $v=0$. In this case, 

$$\sigma_b = \frac{-n+nm}{n+1} = \frac{n(m-1)}{n+1} >0$$

and $\sigma_c = -n$ and $\sigma_b <m$.\\

Hence by Lemma \ref{ConstrLemma3}, the family $\cQ_{0}$ has elements with $\limsup_{q\to \infty} \frac{Q_1(q)}{q}$ ranging through $[-n, 0]$, and Proposition \ref{prop} provides points $\xi \in \bbR^{n \times m}$ satisfying (B) but not (A) with $\hat{\omega}(\xi)$ ranging through the whole interval $[\frac{m}{n}, \infty]$.\\

If $u= \lfloor\frac{m+1}{2}\rfloor < n+1$, consider $v_n$ minimal so that $\sigma_b\ge0$. For any $0\le v \le v_n$, by Lemma \ref{ConstrLemma3} the families $\bQ_{v}$ have 
\begin{equation}\label{arti}
\limsup_{q\to \infty} \frac{Q_1(q)}{q}\textrm{ ranging through }\left[ \frac{-n\lfloor\frac{m+1}{2}\rfloor  +vm }{\lfloor\frac{m+1}{2}\rfloor +v}, \frac{-n + (v + \lfloor\frac{m-1}{2}\rfloor) m }{v+\lfloor\frac{m+1}{2}\rfloor }\right].\end{equation}
 Hence, we construct a family of template satisfying (b) but not (a) with $\limsup_{q\to \infty} \frac{Q_1(q)}{q}$ ranging through

$$ \bigcup_{v=0}^{v_n-1}  \left[ \frac{-n\lfloor\frac{m+1}{2}\rfloor  +vm }{\lfloor\frac{m+1}{2}\rfloor +v}, \frac{-n + (v + \lfloor\frac{m-1}{2}\rfloor) m }{v+\lfloor\frac{m+1}{2}\rfloor }\right] \cup \left[ \frac{-n\lfloor\frac{m+1}{2}\rfloor  +v_nm }{\lfloor\frac{m+1}{2}\rfloor +v_n}, 0 \right] = [-n,0].$$

 Proposition \ref{prop} and \eqref{LimPOmega} provides points $\xi \in \bbR^{n \times m}$ satisfying (B) but not (A) with $\hat{\omega}(\xi)$ ranging through the whole interval $[\frac{m}{n}, \infty]$.\qed\\[3mm]
One can check that the lower contraction rates of the constructed templates are strictly positive, providing positive Hausdorff and Packing dimension.

\begin{remark}\label{rmk1}
Note that when $v\ge1$, property (c) is also satisfied by elements of $\cQ_v$. Hence this construction provides matrices $\bxi \in \bbR^{n \times m}$ satisfying (B) and (C) but not (A) with $\hat{\omega}(\bxi)$ ranging through the whole interval $[\frac{m}{n}, \lfloor \frac{m+1}{2}\rfloor]$.
\end{remark}

\subsection{Construction related to simultaneous properties (B) and (C)}\label{qq4}
 We extend Summer's construction \cite{L} dealing with Property (C) to the case $m>1$. Perhaps surprisingly, this construction also has property (B) if $m>1$.\\

Consider the main parameter $1\le u \le n-1$, and define the two slopes 
\begin{eqnarray}
\sigma_c &:= &\frac{m u - mn}{m+u} > -n, \\
\sigma_b &:= &\frac{m u - (m-1) n}{m+u-1} > -n .
 \end{eqnarray}
 Note that $-n < \sigma_c < \sigma_b <m$ if $m>1$ and $-n < \sigma_c < \sigma_b = m$ if $m=1$.\\
 
For $q_k \gg 1$, consider $p_k \ge q_k + \log q_k$ and a parameter $\lambda_k \in [\sigma_c ,\min(0,\sigma_b)]$ { describing the configuration at $q_k$.}\\

The construction of the elementary temple $R_u(q_k, p_k)$ on the interval $[q_k,p_k]$ is as follows. It is depicted on Figure \ref{FigUnif22}. At $q_k$, we set 
\[R_1(q_k) = \cdots = R_{m+u}(q_k) = \lambda_k q_k \le R_{m+1+u}(q_k) \le \cdots \le R_{m+n}(q_k)\]
where we have enough freedom to ensure that the sum of the components is nul. Then on the interval $[q_k,p_k]$, the $m+u$ first components are equal and share slope $\sigma_c$ while all other have slope $m$. Consider the line $\mathcal{L}$ of slope $-n$ passing through $(p_k, R_{n+m}(p_k))$. Set $q_{k+1}$ to be the abscissa of the intersection of the line $\mathcal{L}$ and the line of slope $\sigma_b$ passing through $(p_k, R_{1}(p_k))$. Since $\sigma_b > -n$, we have $p_k < q_{k+1} < \infty$. On the interval $[p_k,q_{k+1}]$, the $m+u-1$ first components are equal and share slope $\sigma_b$, and among the components $R_{u+m}, \ldots , R_{n+m}$ one follows $\mathcal{L}$ while the others have slope $m$. By definition of $q_{k+1}$, we get
\[\lambda_{k+1} q_{k+1} := R_1(q_{k+1}) = \cdots = R_{m+u}(q_{k+1})  \le R_{m+1+u}(q_{k+1}) \le \cdots \le R_{m+n}(q_{k+1}).\]

Fix $1\le u \le n-1$, choose $q_1\gg1$ and $\lambda_k \in [\sigma_c ,\min(0,\sigma_b)]$, we construct inductively a template $\bR_u((p_k)_{k\in\bN} )$ on the interval $[q_1,\infty)$ by concatenation of elementary templates. Note that $\lambda_{k+1}$ is determined by the choice of $p_k$. We denote $\cR_u$ the set of templates constructed $\bR_u((p_k)_{k\in\bN} )$ for all admissible sequences $(p_k)_{k\in\bN}$. It has the following properties.

\begin{figure}[h!]
 \begin{center}
 \begin{tikzpicture}[scale=0.5]
 
 \draw[black, thick] (0,0) -- (3,-4) node [midway,below] {$\sigma_c$};
 \draw[black, thick] (3,-4) -- (6,-6) node [midway,below] {$\sigma_b$};
 
 \draw[black, thin] (0,1) -- (6,7) node [midway,below] {};
 \draw[black, thin] (0,2) -- (6,8) node [midway,below] {};
 \draw[black, thin] (0,3) -- (3,6) node [midway,below] {};
 \draw[black, thin] (3,6) -- (6,-6) node [midway,above] {};
 \draw[black, thin] (3,-4) -- (6,-1) node [midway,below] {};
 
 \draw(4.2,3) node {$\cL$};
 
  \draw[black, thin, ->] (-1,1.3) -- (14,1.3);
 
   \draw[black, dashed, thin] (0,3) -- (0,-7) node [below] {$q_k$};
   \draw[black, dashed, thin] (3,6) -- (3,-7) node [below] {$p_k$};
  \draw[black, dashed, thin] (6,8) -- (6,-7) node [below] {$q_{k+1}$};

 \end{tikzpicture}
 \end{center}
 \caption{Construction of $\bR(q_k,p_k)$. Slopes different from $-n$ and $m$ are labelled. Thick lines represent more than $1$ component.}\label{FigUnif22}
 \end{figure}
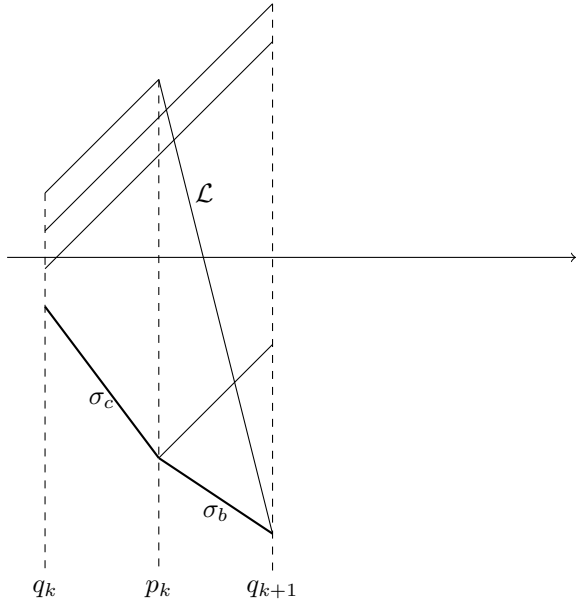

\begin{lemma}\label{rmk3}
Fix $1\le u \le n-1$, for a template $\bR$ in the family $\cR_u$, 
\begin{itemize}
\item $\bR$ does not satisfy (a)
\item $\bR$ satisfies (c)
\item $\bR$ satisfies (b) if $m>1$
\item When $\bR$ ranges through $\cR_u$, then $\limsup_{q\to \infty} \frac{R_{1}(q)}{q}$ ranges through $[ \gamma_u ,\min(0,\sigma_b)]$, where $\gamma_u = \frac{\sigma_c+\sigma_b\left(\frac{m-\sigma_c}{n+\sigma_b}\right)}{1+\frac{m-\sigma_c}{n+\sigma_b}}$.
\end{itemize}
\end{lemma}
The properties regarding (a), (b) and (c) are given by the lower bound of $p_k$ and $-n <\sigma_c < \sigma_b<m$. For the $\limsup$, the upper bound is reached when $p_k = q_k + o(q_k)$ ensuring $\lambda_k = \min(0,\sigma_b) + o(1)$ and the lower bound is reached when $q_k = o(p_k)$: in this case at $p_k$ we have 
$$R_1(p_k)= \sigma_c p_k +o(p_k) \, \textrm{ and } R_{m+n}(p_k)= mp_k +o(p_k)$$
and it follows that $q_{k+1} = \left(1+\frac{m-\sigma_c}{n+\sigma_b}\right)p_k +o(p_k)$ and $R_1(q_{k+1}) = \gamma_u q_{k+1} + o(q_{k+1})$.\\

Observe that $\gamma_1= \frac{m - n \frac{m^2}{m+1}}{1+ \frac{m^2}{m+1}}$ and that $\sigma_b \ge 0$ when $u \ge \frac{n}{2}$. Also, note that the freedom in the construction allows to transform continuously a template from the family $\cR_u$ into one of the family $\bR_{u+1}$ : move one of the free line of slope $m$ down and merge it with the lines of slopes $\sigma_b$ and $\sigma_c$. This preserves properties (b) and (c). Hence, we obtain a family of templates with properties (b) and (c) but not (a) with 
\begin{equation}\label{range}
\limsup_{q\to \infty} \frac{R_{1}(q)}{q}
\,\,\,
\text{ ranging through} 
\,\,\, [ \gamma_1 ,0].
\end{equation}
Proposition \ref{prop} provides points $\xi \in \bbR^{n \times m}$ satisfying (B) and (C) but not (A). Now condition \eqref{range} together with the first formula from \eqref{LimPOmega} show that $\hat{\omega}(\xi)$ can attain any value from the whole interval $[\frac{m}{n},\frac{m^2}{m+1}]$.\\
 \qed
 
 In particular, it covers $m=2$ and $\bxi$ with property (B). 
 
 Note that the constructed templates have strictly positive lower contraction rate, providing strictly positive Hausdorff and Packing dimensions - and thus existence.
 


\appendix

\section{About  monotonicity lemma \ref{monotonicity}}

In this appendix, we adapt \cite[Lemma 3]{Szeged} to the setting of sup-norms to obtain Lemma \ref{monotonicity}.\\
Recall that $m=2$ and $n\ge 2 $ is arbitrary. We set $d = 2+n$. We consider a matrix $\pmb{\xi} \in \bbR^{{n\times 2}}$. According to Theorem \ref{normindep}, we can prove the statement for the Euclidian norm $\|.\|_2$.\\
 
Consider the sequence of best approximation related to the Euclidian norm 
 $\bx_i = (\pmb{x}_i,\pmb{y}_i)$ ---  $i$-th best approximation vector in Euclidean norm,
  $\pmb{x}_i = (x_{1,i}, x_{2,i}) \in \mathbb{Z}^2, \pmb{y}_i = (y_{1,i},...,y_{n,i}) \in \mathbb{Z}^n$, 
    $L_i$ - the value of $i$-th best approximation 
    (we assume that both norms, to estimate the value of $\pmb{x}$  in $\mathbb{R}^2$ and $\pmb{\xi}\pmb{x} - \pmb{y}$  in $\mathbb{R}^n$, are  Euclidean. So
    $$
    L_i = {\|}\pmb{\xi}\pmb{x}_i - \pmb{y}_i{\|}_2.
    $$
    We deal with the two-dimensional subspace
    $$\frak{L} =\{ \,\,\bx\in \mathbb{R}^d: \pmb{y} =\pmb{\xi}\pmb{x}\,\,\}.
    $$
   We use notation
 $Z_i ={\|}\bx_i{\|}_2$ - Euclidean norm of $\bx_i\in \mathbb{Z}^d$,  $X_i= {\|}\pmb{x}_i{\|}_2 $ - Euclidean norm of $\pmb{x}_i \in \mathbb{Z}^2$. We cannot claim that $Z_i\le Z_{i+1}$ but
 it is clear that 
 \begin{equation}\label{zs1}
 Z_i \ll_{\pmb{\xi}}X_i,\,\,\,\, \forall  \, i \in \bN
 \end{equation}
 and
  \begin{equation}\label{zs10}
 L_i 
 \asymp_{\pmb{\xi}}  \zeta_i = \,\,\text{Euclidean distance from}\,\,\bx_i\,\,\text{to}\,\, \frak{L},
 \,\,\,\, \forall  \, i \in \bN
 \end{equation}
 In addition, we should note that 
   \begin{equation}\label{zs11}
 Z_i\ll_{\pmb{\xi}} X_i< X_{i+1}\le Z_{i+1}. 
 \end{equation}

Furthermore, here in Appendix  we everywhere assume that 
  \begin{equation}\label{zs2}
 \max \left( \frac{X_{i+1}}{X_i}, 
 { \frac{L_i}{L_{i+1}}
 \, 
 } \right)\le M,\,\,\, \forall \, i \in \bN.
 \end{equation}
By the exponentially growth of best approximations (See Bugeaud - Laurent \cite{MMJ}),
$$
\forall \, C>1\,\, \exists C_1(C) \,\,\,\text{such that}\,\,\, X_{i+C_1(C)} \ge C X_i,
$$
so by  \eqref{zs1} we conclude that 
  \begin{equation}\label{zs4}
\exists C_2 = C_2(M,\pmb{\xi}) \,\,\,\text{such that}\,\,\, X_{k} \ge Z_\nu,\,\,\ \forall \, k \ge \nu +  C_2.
\end{equation}

Now we give a precise formulation of Lemma \ref{monotonicity}.

\begin{lemma} \label{monotonicity1}
Suppose  
\eqref{zs2} is satisfied and 
that all the best approximations 
$$
\bx_\nu, \bx_{\nu+1},....,\bx_{k-1}, \bx_k
$$
belong to the same two-dimensional linear subspace $\pi$.
Then
 \begin{equation}\label{zs3}
 L_\nu X_{\nu+1}\ll_{M,\pmb{\xi}} L_{k-1}X_k.
 \end{equation}
\end{lemma}

A proof below is not self-contained. It relies on the argument of Lemmas 2 and 3  from \cite{Szeged}.
Local notation and numeration of formulas (28,29,30,31) below corresponds to \cite{Szeged}.

The only problem is that everywhere in our exposition we use $\bx_i$  for the vectors of best approximations, while for the spherical best approximation vectors in  \cite{Szeged} the notation  $\pmb{z}_j$ was used.
As we use here other notation from \cite{Szeged} and refer to several inequalities from  \cite{Szeged}, it seems convenient to use in the rest of the proof $\pmb{z}_j=\bx_j$.

\vskip+0.3cm
Proof.

\vskip+0.1cm

Case 1$^0$. If $k<\nu+ 2C_2$ then \eqref{zs3} follows from  \eqref{zs2} applied $2C_2$ times.

\vskip+0.1cm

Case 2$^0$. If $k\ge \nu+ 2C_2$,
first of all we define $ k_1 = k -C_2$.
Then \eqref{zs2}  ensures 
$$
L_{k_1-1}X_{k_1} \asymp_{M,{\pmb{\xi}}} L_{k-1}X_k.
$$
So it is enough to prove \eqref{zs3} for $k_1$ instead of $k$. 

As now $k_1 \ge \nu+C_2$, 
we can apply (\ref{zs4}) to see that 
$Z_{k_1} \ge X_{k_1}\ge Z_{\nu+1}$. Now Lemma 2 \cite{Szeged} ensures (31) in the same notation that in 
\cite{Szeged}. 
That is $\xi_i$ is the Euclidean distance from $\pmb{z}_i =\bx_i$ to the line $\ell$ which minimise the angle between vectors in $\frak{L}$
and $\pi$.
Now $\Xi_l$ from the proof from \cite{Szeged}
should be the minor axis of the ellipse
$$\mathcal{G}_l=
\pi \cap \{ \bx:  \|\pmb{\xi}\pmb{x} - \pmb{y}\|_2\le L_{l+1}\}.
$$
To deduce \eqref{zs3} for the best approximation in the sense of the present paper (but not for the spherical best approximation from \cite{Szeged}) we should establish the following inequalities  from \cite{Szeged}:

\vskip+0.1cm
$\bullet$ (28) for $ \nu \le l \le k_1$;

$\bullet$ (29)  with $ k_1$ instead of $k$;

$\bullet$ 
(30)  for $ \nu \le l \le k_1$, 

\vskip+0.1cm

with constants which may differ from those in \cite{Szeged}.
 
First of all, observe that 
  (28) is obvious. 
  
  As for (29), it remains valid, probably with some positive constants depending on $\pmb{\xi}$ instead of absolute constants. The explanation is as follows.   Consider arbitrary $l$ from the range
  $\nu\le l\le k$. The two-dimensional convex body
$$
\mathcal{E}_l =\{{\bx}\in 
\mathcal{G}_l 
:\,
\|{\bx}\|_2\le Z_{l+1}\} \supset \pmb{z}_l, \pmb{z}_{l+1}
$$
does not contain non-trivial integer points  inside.
It is clear that 
$$
{\rm vol}_2\, \mathcal{E}_l \le 4\Xi_lZ_{l+1}
$$
(recall that $\Xi_l$ is the half of minor axis of ellipse $\mathcal{G}_l$).
As for the lower bound,
we know that $ \pmb{z}_l \in \mathcal{G}_l$, so
\begin{equation}\label{zs33}
\Xi_l\le Z_{l}
\ll_{\pmb{\xi}} Z_{l+1}
\end{equation}
  by (\ref{zs11}). If  $
\Xi_l\le Z_{l+1}
$ then
$$
 {\rm vol}_2\, \mathcal{E}_l\ge \frac{\Xi_lZ_{l+1}}{2}.
$$
If
$
\Xi_l> Z_{l+1}$  then $\mathcal{E}_l$ is a disk of radius $Z_{l+1}$. Then  by (\ref{zs33}) we get
$$  {\rm vol}_2\, \mathcal{E}_l = \pi Z_{l+1}^2\gg_{\pmb{\xi}} {\Xi_lZ_{l+1}}.
$$
So in any case 
$$
 {\rm vol}_2\, \mathcal{E}_l\asymp_{\pmb{\xi}} \Xi_lZ_{l+1}.
$$
Now by Minkowski convex body theorem 
 for the lattice $ \Lambda =\langle \pmb{z}_l, \pmb{z}_{l+1}\rangle_\mathbb{Z}\subset \pi$ we have 
$${
\rm det}\, \Lambda \asymp_{\pmb{\xi}} {\rm vol}_2\, \mathcal{E}_l  \,\,\,\, \nu\le l \le k-1.
$$
This ensures inequality (29) from \cite{Szeged}.

To get (30) we need to generalize geometric argument from \cite{Szeged}.

Here it will be important that $l\le k_1$.
For such $l$ we consider two positive and small constants $c_1,c_2$.
Let $H_l$ be the half of the length of the major axis of ellipse $\mathcal{G}_l$
consider domains
$$
\mathcal{A}_l = \{\bx \in\mathcal{G}_l: \|\bx\|_2
\le c_1H_l\},\,\,\,\,
\mathcal{B}_l = \{\bx\in\mathcal{G}_l: 
\|\bx\|_2
\ge (1-c_2)H_l\}.
$$
If $c_1 $   and $c_2$ are small enough we have the following {\it property}.
If $\bx' $ and $\bx''$ belong to the same connectedness component of $\mathcal{B}_l$ then $ \bx'-\bx''$ belongs to $\mathcal{A}$.
   
Now we claim that for $ l\le k_1$ we have 
\begin{equation}\label{zs000}
\pmb{z}_l \not\in \mathcal{B}_l.
\end{equation}   
Indeed,
if $\pmb{z}_l \in \mathcal{B}_l$,
then by (\ref{zs4}) we have
$Z_k \ge Z_l \ge (1-c_2) H_l$ 
and so $\pmb{z}_k \in \mathcal{B}_l$.
By the {\it property } described above, the integer point 
$\pmb{z}_k -\pmb{z}_l$ belongs to $\mathcal{A}_l$.
But this is not possible for small $c_1$ because 
 $ \mathcal{A}_l $ lies strictly inside the set $ \mathcal{E}_l$.

Now we see that (\ref{zs000}) is valid and this immediately gives 
(30) from \cite{Szeged} with $ a = 1-\sqrt{1-c_2^2}$. $\Box$

\end{document}